\documentclass{amsart}
\usepackage[T1]{fontenc}
\usepackage{latexsym,amssymb,amsmath}

\usepackage[all,ps]{xy}
\usepackage{amsthm}
\usepackage{longtable}
\usepackage{epsfig}
\usepackage{mathrsfs}
\usepackage{hhline}
\usepackage{epic}

 \newcommand{\Galg}{\mathbf{G}}
 
 \newcommand{\Balg}{\mathbf{B}}

 \newcommand{\Ualg}{\mathbf{U}}

 \newcommand{\Halg}{\mathbf{H}}
 \newcommand{\Lalg}{\mathbf{L}}
 \newcommand{\Talg}{\mathbf{T}}
 
 \newcommand{\Xalg}{\mathbf{X}}

 \newcommand{\al}{\mathscr{C}}

 \newcommand{\rk}{\operatorname{Rk}}
 \newcommand{\ad}{\operatorname{ad}}
 \newcommand{\rad}{\operatorname{Rad}}
 \newcommand{\Stab}{\operatorname{Stab}}

 \newcommand{\Ind}{\operatorname{Ind}}
 \newcommand{\Res}{\operatorname{Res}}
 \newcommand{\Nn}{\operatorname{N}}

 \newcommand{\Cen}{\operatorname{C}}

 \newcommand{\ord}{\operatorname{Ord}}

 \newcommand{\vol}{\operatorname{Vol}}
 \newcommand{\cprod}{\centerdot}

 \newcommand{\semi}[1]{\rtimes\langle\,#1\,\rangle}
 
 \newcommand{\cyc}[1]{\langle\,#1\,\rangle}

 \newcommand{\C}{\mathbb{C}}
 \newcommand{\F}{\mathbb{F}}
 
\newcommand{\N}{\mathbb{N}}
 \newcommand{\Q}{\mathbb{Q}}
 \newcommand{\Z}{\mathbb{Z}}

 \newcommand{\Irr}{\operatorname{Irr}}

 \newcommand{\Zz}{\operatorname{Z}}
 
 \newcommand{\Sc}{\operatorname{sc}}
\newcommand{\om}{\varpi}
\newcommand{\cal}[1]{\mathcal{#1}}

\newtheorem{theorem}{Theorem}[section] 
\newtheorem{lemma}[theorem]{Lemma}     
\newtheorem{corollary}[theorem]{Corollary}
\newtheorem{proposition}[theorem]{Proposition}

\newtheorem{convention}[theorem]{Convention}
\theoremstyle{definition}
\newtheorem{remark}[theorem]{Remark}

\title[]
{On semisimple classes and semisimple characters in finite reductive groups}
\author{Olivier Brunat}

\address{Ruhr-Universit\"at Bochum\\
Fakult\"at f\"ur Mathematik\\
D-44780 Bochum, Germany\\}
\email{Olivier.Brunat@ruhr-uni-bochum.de}


\begin{document}

\begin{abstract} 
In this article,  we study the elements with disconnected centralizer in
the Brauer complex associated to a simple algebraic group $\Galg$
defined over a finite field with corresponding Frobenius map $F$ and
derive the number of $F$-stable semisimple classes of $\Galg$ with
disconnected centralizer when the order of the fundamental group has
prime order. We also discuss extendibility of semisimple characters to
their inertia group in the full automorphism group.  As a consequence, we
prove that ``twisted'' and ``untwisted''  simple groups of type $E_6$
are ``good'' in defining characteristic, which is a contribution to the
general program initialized by Isaacs, Malle and Navarro to prove the
McKay Conjecture in representation theory of finite groups. 
\end{abstract}
\maketitle

\section{Introduction}\label{intro}

This article is concerned with the semisimple characters of finite
reductive groups. A finite reductive group is the fixed-point subgroup
$\Galg^F$ of a connected reductive group $\Galg$ defined over the
finite field $\F_q$ of characteristic $p>0$, where
$F:\Galg\rightarrow\Galg$ is the Frobenius map corresponding to this
$\F_q$-structure. The semisimple characters of $\Galg^F$ are the
constituents of the duals of Gelfand-Graev characters (for the
Alvis-Curtis duality) and play an important role in the ordinary
representation theory of $\Galg^F$, because, apart from a few
exceptions, they are the $p'$-characters of $\Galg^F$ (that is the
irreducible characters of $\Galg^F$ whose degree is prime to $p$). In
the following, we will write $\Irr_s(\Galg^F)$ for the set of
semisimple characters of $\Galg^F$.  One of the aims of this work is to
study these characters, compute their number, understand the action of
the automorphism group of $\Galg^F$ on $\Irr_s(\Galg^F)$, and determine the
extendibility of $\chi\in\Irr_s(\Galg^F)$ to
its inertia group in the full automorphism group.  These questions
are crucial, for example in order to prove that $\Galg^F$ satisfies
the inductive McKay condition at the prime $p$.


Using  Deligne-Lusztig theory~\cite{DeligneLusztig}, Lusztig has shown
that the irreducible characters of $\Galg^F$ can be partitioned into
series (the so-called rational Lusztig series) labelled by the
semisimple classes of $\Galg^{*F^*}$, where $(\Galg^*,F^*)$ denotes a
pair dual to $(\Galg,F)$.  If such a series is labelled by a semisimple
class of $\Galg^{*F^*}$ with representative $s$, then it contains
$|A_{\Galg^*}(s)^{F^*}|$ semisimple characters, where
$A_{\Galg^*}(s)=\Cen_{\Galg^*}(s)/\Cen_{\Galg^*}^{\circ}(s)$ is the
component group of $s$. In fact, $\Irr_s(\Galg^F)$ can be parametrized
in a natural way by pairs $(s,\xi)$ where $s$ runs over a set of
representatives of the semisimple classes of $\Galg^{*F^*}$ and
$\xi\in\Irr(A_{\Galg^*}(s)^{F^*})$.  So, in order to understand
$\Irr_s(\Galg^F)$, we have to particularly consider the semisimple
classes of $\Galg^{*F^*}$ and their component groups.


In~\cite{Br8}, we explicitly computed the number of semisimple classes of
$\Galg^{*F^*}$ when $\Galg^*$ is simple and $p$ is a good prime for
$\Galg^*$.
For that, we used the theory of Gelfand-Graev characters for connected
reductive groups with disconnected center, developed by
Digne-Lehrer-Michel in~\cite{DLM1} and \cite{DLM2}.  This method gives
a lot of information on the $F^*$-stable semisimple
$\Galg^*$-classes of $\Galg^*$ with disconnected centralizer (by
the centralizer of an $F^*$-stable class, we mean the centralizer of a
fixed $F^*$-stable representative), which allows us to prove the
inductive McKay condition in defining characteristic for simple groups
coming from simple algebraic groups with fundamental group of order
$2$; see~\cite[Theorem 1.1]{BrHi}.  However, we cannot derive
from~\cite{Br8} all information that we need, especially phenomena
appearing only in the algebraic group, as for example the description of the
$F^*$-stable semisimple classes with disconnected centralizer
such that the fixed-point subgroup $A_{\Galg^*}(s)^{F^*}$ is trivial ($s$
is any $F^*$-stable representative); see
Remark~\ref{rk:manque}.

In this work, we will consider the Brauer complex, initially introduced by
J. Humphreys in \cite{Hum} for describing $p$-modular representation theory of
$\Galg^F$. When $\Galg$ is a simple simply-connected group, Deriziotis
proved in~\cite{DeriBrauer} that the $F$-stable semisimple classes of
$\Galg$ (and thus, the semisimple classes of $\Galg^F$) are parametrized
by the faces of the Brauer complex of maximal dimension.
We generalize here some results to any simple algebraic groups using
an approach of Bonnaf\'e~\cite{bonnafequasi}. 

Moreover, we are interested in the problem of the extendibility of
semisimple characters of $\Galg^F$ to their inertia groups in
$\operatorname{Aut}(\Galg^F)$.
Digne-Michel~\cite{DMnonconnexe} and Malle~\cite{MalleNonConnexe}
developed a theory of Deligne-Lusztig characters for finite disconnected
reductive groups.  Using this theory, Sorlin~\cite{Sorlin}
constructed extensions of Gelfand-Graev characters of $\Galg^F$ to
$\Galg^F\semi\sigma$, where $\sigma$ is a quasi-central
semisimple or unipotent automorphism of $\Galg$. We will use
results of~\cite{Sorlin} in order to prove that, under certain
assumptions, the semisimple characters of $\Galg^F$ are extendible to
their inertia groups in the full automorphism group.

Finally, recall that the McKay Conjecture asserts that for any finite
group $G$, if $\Irr_{p'}(G)$ denotes the set of $p'$-characters of $G$,
then $|\Irr_{p'}(G)|=|\Irr_{p'}(\Nn_G(P))|$, where $P$ is a fixed
$p$-Sylow subgroup of $G$. In~\cite{IMN}, Isaacs-Malle-Navarro proved a
reduction theorem of this conjecture to finite simple groups. They
showed
that if every simple group satisfies a refined property, the so-called
inductive McKay condition (see~\cite[\S10]{IMN} for more details), then
the McKay Conjecture holds for all finite groups.  As an application of our
results, we will prove that ``twisted'' and ``untwisted'' finite simple
groups of type $E_6$ satisfy the inductive McKay condition in defining
characteristic. 

The paper is organized as follows. In Section~\ref{section1}, we
introduce the Brauer complex of $\Galg$ and describe the faces
containing points with disconnected centralizer; see
Theorem~\ref{alovefixe}. Then we compute the number of $F$-stable
semisimple classes of $\Galg$ with disconnected centralizer when
$\Galg$ is not of type $D_{2n}$ and has fundamental group of prime
order; see Proposition~\ref{prop:premier}. Note that this result
requires no condition on $q$. Furthermore, if $\Galg$ is not of type
$D_{2n}$, we describe the $F$-stable points of the Brauer complex in the
case that $F$ acts trivially on the center of $\Galg$; see
Propositions~\ref{prop:qmoinsun} and~\ref{prop:qplusun}.  As first
consequences we prove that if $p$ is odd, then the McKay Conjecture holds 
for $\Galg^F$ at the prime $p$, where $\Galg$ is a simple and
simply-connected group of type $D_{2n+1}$; see
Remark~\ref{rk:McKayDn}. It also holds for $p=2$ (resp. $p=3$) 
when $\Galg$ is a simple and simply-connected
group of type $E_6$ (resp. $E_7$); see
Remark~\ref{rk:McKayE67}.  In Section~\ref{section2}, we recall the
construction of semisimple characters and give the action of
automorphisms of $\Galg^F$ on $\Irr_s(\Galg^F)$. Then we discuss
extendibility in a special case of these characters to their inertia
groups; see Proposition~\ref{extension}.  Finally, in
Section~\ref{section3} we prove the inductive McKay condition in
defining characteristic (for $p>3$) for ``untwisted'' simple groups of
type $E_6$ in Theorem~\ref{E6bon}, and for the ``twisted'' version in
Theorem~\ref{2E6bon}.
Note that these methods also prove that the inductive McKay condition
is satisfied at the prime $p$ 
by simple groups of Lie type of type $A_{2n}$ such that
$2n+1$ is a prime number distinct from $p$; see Proposition~\ref{An}.



\section{Semisimple classes with disconnected
centralizers}\label{section1}
\subsection{Notation}\label{notation}
Throughout this paper, $\Galg$ denotes a simple algebraic group over
$\overline{\F}_p$.
We fix a maximal torus $\Talg$
contained in a Borel subgroup $\Balg$ of $\Galg$. Let $\Phi$
be the root system of $\Galg$ relative to $\Talg$. We write $\Phi^+$ and
$\Delta$ for the set of positive roots and the set of simple roots of
$\Phi$ corresponding to $\Balg$. We denote by $X(\Talg)$ and $Y(\Talg)$
the groups of characters and cocharacters of $\Talg$ and
write $\cyc{,}:X(\Talg)\times Y(\Talg)\rightarrow \Z$ for the duality
pairing between $X(\Talg)$ and $Y(\Talg)$. For $\alpha\in \Phi$, we
denote by $\alpha^{\vee}$ 
the coroot of $\alpha$ and write
$\Phi^{\vee}=\{\alpha^{\vee}\, |\, \alpha\in\Phi\}$.  We define
$V=\Q\otimes_{\Z} Y(\Talg)$ and $V^*=\Q\otimes_{\Z} X(\Talg)$, and extend
$\cyc{,}$ to a nondegenerate bilinear form $$\cyc{,}:V^*\times
V\rightarrow\Q.$$ Let $\Galg_{\Sc}$ (resp. $\Galg_{\ad}$) be the
simply-connected version (resp. adjoint version) of $\Galg$. We denote
by $\pi_{\Sc}:\Galg_{\Sc}\rightarrow\Galg$ and
$\pi_{\ad}:\Galg\rightarrow\Galg_{\ad}$ corresponding isogenies.
Write $\Talg_{\Sc}=\pi_{\Sc}^{-1}(\Talg)$ and
$\Talg_{\ad}=\pi_{\ad}(\Talg)$.  Then $\Talg_{\Sc}$ and $\Talg_{\ad}$
are maximal tori of $\Galg_{\Sc}$ and $\Galg_{\ad}$,
and the surjective homomorphisms
$\pi_{\Sc}:\Talg_{\Sc}\rightarrow\Talg$ and
$\pi_{\ad}:\Talg\rightarrow\Talg_{\ad}$ induce injective homomorphisms
$\pi_{\Sc,X}:X(\Talg)\rightarrow X(\Talg_{\Sc}),\,\chi\rightarrow
\chi\circ\pi_{\Sc}$ and
$\pi_{\ad,X}:X(\Talg_{\ad})\rightarrow
X(\Talg),\chi\rightarrow\chi\circ\pi_{\ad}$. 
Using $\pi_{\Sc,X}$ and $\pi_{\ad,X}$, we identify $X(\Talg)$ with a
subgroup of $X(\Talg_{\Sc})$ containing $X(\Talg_{\ad})$, such that the
root systems $\pi_{\Sc}(\Phi)$ and $\pi^{-1}_{\ad}(\Phi)$ of
$\Galg_{\Sc}$ and $\Galg_{\ad}$ are identified with $\Phi$.  
Similarly, using the injective morphisms $\pi_{\Sc,Y}:Y(\Talg_{\Sc})\rightarrow
Y(\Talg),\,\gamma\rightarrow\pi_{\Sc}\circ\gamma$ and
$\pi_{\ad,Y}:Y(\Talg)\rightarrow Y(\Talg_{\ad}),\,\gamma\rightarrow
\pi_{\ad}\circ\gamma$ induced by $\pi_{\Sc}$
and $\pi_{\ad}$, the group $Y(\Talg)$ is viewed as a subgroup of
$Y(\Talg_{\ad})$ containing $Y(\Talg_{\Sc})$ such that
$\pi_{\Sc,Y}^{-1}(\Phi^{\vee})$ and $\pi_{\ad,Y}(\Phi^{\vee})$ are
identified with $\Phi^{\vee}$. Note that $V=\Q\otimes_{\Z}
Y(\Talg_{\Sc})=\Q\otimes_{\Z} Y(\Talg_{\ad})$, $V^*=\Q\otimes_{\Z}
X(\Talg_{\Sc})=\Q\otimes_{\Z} X(\Talg_{\ad})$, and the linear maps
$\pi_{\Sc,X}$ and $\pi_{\Sc,Y}$ (resp. $\pi_{\ad,X}$ and $\pi_{\ad,Y}$)
are adjoint maps with respect to $\cyc{,}$.  
We define the group of weights by 
\begin{equation}
\Lambda=\{\lambda\in
V^*\,|\,\cyc{\lambda,\alpha^{\vee}}\in\Z\quad\forall\,\alpha\in\Phi\}.
\label{eq:weightlattice}
\end{equation}
Recall that $X(\Talg_{\ad})=\Z\Phi$ and $X(\Talg_{\Sc})=\Lambda$, and
the fundamental group of $\Phi$ is the finite group
$\Lambda/\Z\Phi=X(\Talg_{\Sc})/X(\Talg_{\ad})$.
Now, we denote by $(\om_{\alpha}^{\vee})_{\alpha\in\Delta}$ and 
$(\om_{\alpha})_{\alpha\in\Delta}$ the dual bases with respect to
$\cyc{,}$ of $\Delta$ and $\Delta^{\vee}=\{\alpha^{\vee}\,|\,\alpha\in\Delta\}$,
respectively. Since $X(\Talg_{\ad})=\Z\Phi$ and
$X(\Talg_{\Sc})=\cyc{\om_{\alpha},\,\alpha\in\Delta}_{\Z}$, we deduce
that
\begin{equation}
Y(\Talg_{\Sc})=\bigoplus_{\alpha\in\Delta}\Z\alpha^{\vee}\quad\textrm{and}
\quad Y(\Talg_{\ad})=\bigoplus_{\alpha\in\Delta}\Z\,\om_{\alpha}^{\vee}.
\label{eq:yad}
\end{equation}
We denote by $W=\Nn_{\Galg}(\Talg)/\Talg$ the Weyl group
of $\Galg$. Then $W$ acts on $X(\Talg)$ and on $Y(\Talg)$ by
\begin{equation}
\label{eq:Waction}
^w\chi(t)=\chi(t^w)\quad\textrm{and}\quad
\gamma^w(t)=\gamma(t)^w,
\end{equation}
for $\gamma\in Y(\Talg)$, $\chi\in X(\Talg)$ and $t\in\Talg$. In
particular, we have $\cyc{^w\chi,\gamma}=\cyc{\chi,\gamma^w}$ and
$W(\Phi)=\Phi$ and $W(\Phi^\vee)=\Phi^\vee$. Recall
from~\cite[\S1.9]{carter2} that
$W=\cyc{s_{\alpha}\,|\,\alpha\in\Phi}$, and for $\chi\in X(\Talg)$ and
$\gamma\in Y(\Talg)$, we have
$s_{\alpha}(\chi)=\chi-\cyc{\chi,\alpha^\vee}\alpha$ and
$s_{\alpha}(\gamma)=\gamma-\cyc{\alpha,\gamma}\alpha^{\vee}$.

\subsection{Semisimple classes with disconnected centralizers}
We denote by $\Q_{p'}$ the additive subgroup of $\Q$ of rational numbers of the
form $a/b$ with $b$ not divisible by $p$.  We choose an isomorphism of
groups $\tilde{\iota}:\Q_{p'}/\Z\rightarrow \overline{\F}_p^{\times}$ as
in~\cite[3.1.3]{carter2}. Moreover, following~\cite[3.1.2]{carter2},  we
identify the group $\Q_{p'}/\Z\otimes_{\Z} Y(\Talg)$ with $\Talg$ using the
isomorphism of groups $\tilde{\iota}:\Q_{p'}/\Z\otimes_{\Z}Y(\Talg)\rightarrow
\Talg,\,r\otimes \gamma\mapsto \gamma(\widetilde{\iota}(r))$.
Furthermore, there is a surjective homomorphism from
$\Q_{p'}\otimes_{\Z} Y(\Talg)$ onto $\Q_{p'}/\Z\otimes_{\Z} Y(\Talg)$
with kernel $Y(\Talg)$, which induces the isomorphism of groups
\begin{equation}
\Talg\simeq \Q_{p'}/\Z\otimes_{\Z} Y(\Talg)\simeq \Q_{p'}\otimes_{\Z}
Y(\Talg)/Y(\Talg).
\label{eq:tore}
\end{equation}
Note that the action on $W$ on $Y(\Talg)$ defined in
Equation~(\ref{eq:Waction}) can be naturally extended to
$\Q_{p'}/\Z\otimes_{\Z} Y(\Talg)$ and to $(\Q_{p'}\otimes_{\Z}
Y(\Talg))/Y(\Talg)$ and is compatible with the isomorphisms of
Equation~(\ref{eq:tore}).
We define $\overline{W}_a= Y(\Talg) \rtimes W$.
Note that $\overline{W}_a$ acts on $V$ as a group of affine
transformations by $$(\gamma\cdot
w)(\lambda\otimes\gamma')=\lambda\otimes\gamma'^w+\gamma,$$ for
$\gamma,\,\gamma'\in Y(\Talg)$, $\lambda\in\Q$ and $w\in W$.
So, we have $$ \left(
\Q_{p'}\otimes_{\Z} Y(\Talg)/Y(\Talg)\right)/W=\Q_{p'}\otimes_{\Z}
Y(\Talg)/\overline{W}_a.$$ Since the set of semisimple
classes $s(\Galg)$ of $\Galg$ is in bijection with the set $\Talg/W$ of
$W$-orbits on $\Talg$ (see~\cite[3.7.1]{carter2}), we deduce that
$s(\Galg)$
is in bijection with 
\begin{equation}
\label{eq:isoclass}
(\Q_{p'}\otimes_{\Z} Y(\Talg))/\overline{W}_a,
\end{equation}
  Now, we
write $\alpha_0$ for the highest root of $\Phi$ (with respect to the
height defined by $\Delta$) and put
\begin{equation}
\alpha_0=\sum_{\alpha\in\Delta}n_{\alpha}\,\alpha,
\label{eq:heightelement}
\end{equation}
with $n_{\alpha}\in\N^*$. 
We define the affine Weyl group $W_a$ of $V$ as the subgroup
of affine transformations of $V$ generated by 
$s_{\alpha}$ (for $\alpha\in\Delta$) and 
$s_{\alpha_0,1}=s_{\alpha_0}+\alpha_0^{\vee}$.
Then by~\cite[p.174]{Bourbaki456}, the alcove
$$\al=\{\lambda\in V\ |\ \cyc{\alpha,\lambda}\geq 0\ \forall
\alpha\in\Delta,\,\cyc{\alpha_0,\lambda}\leq 1\}$$
is a fundamental domain for the action of $W_a$ on $V$. 
Recall that
$W_a=Y(\Talg_{\textrm{sc}})\rtimes W$ (see~\cite[VI\S2,\,Prop
1]{Bourbaki456}).  In particular, $W_a\leq \overline{W}_a$, which implies that
every $\overline{W}_a$-orbit of $V$ contains at least one element of
$\al$.
We write 
$$\widetilde{\Delta}=\Delta\cup\{-\alpha_0\}$$for the affine Dynkin
diagram of $\Galg$ and
$\widetilde{\Delta}_{\min}=\{\alpha\in\widetilde{\Delta}\,|\,n_{\alpha}=1\}$,
with the convention that $n_{-\alpha_0}=1$.
For $\alpha\in\widetilde{\Delta}_{\min}$, we set
$z_{\alpha}=w_{\alpha}w_{0}$, where $w_0$ and $w_{\alpha}$ are the
longest elements of $W$ and $W_{\Delta\backslash\{\alpha\}}$,
respectively (note that $w_{-\alpha_0}=w_0$ and $z_{-\alpha_0}=1$).
Then $z_{\alpha}$ induces an automorphism
of the extended Dynkin diagram $\widetilde{\Delta}$. We define 
$$\cal A=\{z_{\alpha}\,|\,\alpha\in
\widetilde{\Delta}_{\min}\}.$$
Recall that $\cal A$ is isomorphic to the center
$\operatorname{Z}(\Galg_{\Sc})$ as follows. 
By~\cite[Corollaire p.177]{Bourbaki456}, 
we have
$\operatorname{Z}(\Galg_{\Sc})=\{\tilde{\iota}(\om_{\alpha}^{\vee})\,|\,
\alpha\in\widetilde{\Delta}_{\min}\}$. Moreover, for
$z\in\operatorname{Z}(\Galg_{\Sc})$, there is a unique element $w_z\in
W$ (obtained as the projection on $W$ of any element $\omega$ of $W_a$
satisfying $\omega(\al)=y+\al$ for $y\in \Q_{p'}\otimes Y(\Talg)$ such
that $\tilde{\iota}(y)=z$; note that $w_z$ is well-defined because it
does not depend on $y$).  Then \cite[Proposition 6 p.176]{Bourbaki456}
implies that 
$z_{\alpha}=w_{\tilde{\iota}(\om_{\alpha}^\vee)}$
for $\alpha\in\widetilde{\Delta}_{\min}$. Furthermore, the map
\begin{equation}
\operatorname{Z}(\Galg_{\Sc})\rightarrow \cal A,\quad
\widetilde{\iota}(\om_{\alpha}^{\vee})\mapsto z_{\alpha},
\label{eq:isocentre}
\end{equation}
is an isomorphism of groups; see~\cite[VI.\S2.3]{Bourbaki456}.

Now, we write $\widehat{W}_a=\cyc{Y(\Talg_{\ad}),W}$ and
$\Gamma_{\al}$ for the subgroup of $\widehat{W}_a$ which stabilizes
$\al$. Then in~\cite[VI,\S2\, Prop 6]{Bourbaki456}, the following
result is proven.
\begin{proposition}
\label{DescInvariantAlcove}
The group of automorphisms of $\widetilde{\Delta}$ induced by
elements of $W$ is $\cal A$. For $\alpha\in\widetilde{\Delta}_{\min}$,
we set $f_{\alpha}=z_{\alpha}+\om_{\alpha}^\vee$. Then we have
$$\Gamma_{\al}=\{f_{\alpha}\,|\,\alpha\in\widetilde{\Delta}_{\min}\}.$$
Moreover, the map
$$\om^\vee:\cal A\rightarrow Y(\Talg_{\ad})/Y(\Talg_{\Sc}),\,z_{\alpha}\mapsto
\om_{\alpha}^\vee+Y(\Talg_{\Sc}),$$ is an isomorphism of groups.
\end{proposition}

Note that, by composition of $\om^{\vee}$ with the isomorphism
defined in Equation~(\ref{eq:isocentre}), we can identified the
quotient $Y(\Talg_{\ad})/Y(\Talg_{\Sc})$ with
$\operatorname{Z}(\Galg_{\Sc})$.

For $\lambda\in V$, we will denote by
$(\lambda_{\alpha})_{\alpha\in\widetilde{\Delta}}$ its affine
coordinates, that is, the unique family of rational numbers such that
$\sum_{\alpha\in\widetilde{\Delta}}\lambda_{\alpha}=1$ and
$\lambda=\sum_{\alpha\in\widetilde{\Delta}}\frac{\lambda_{\alpha}}{n_{\alpha}}\,\om_{\alpha}^{\vee}$,
where $n_{\alpha}$ are the integers defined in
Equation~(\ref{eq:heightelement}) and $\om_{-\alpha_0}^{\vee}=0$.
Note that $\lambda\in\al$ if and only if $\lambda_{\alpha}\geq 0$ for
every $\alpha\in\widetilde{\Delta}$; see~\cite[Corollaire
p.175]{Bourbaki456}.

Now, following~\cite{bonnafequasi}, we define the subgroup $\cal
A_{\Galg}$ of $\cal A$ to be the inverse image of
$Y(\Talg)/Y(\Talg_{\Sc})$ under $\om^{\vee}$. We also define
\begin{equation}
\widetilde{\Delta}_{\min,\Galg}=\{\alpha\in\widetilde{\Delta}_{\min}\,|\,
z_{\alpha}\in\cal A_{\Galg}\}\quad\textrm{and}\quad
\Gamma_{\Galg}=\{f_{\alpha}\,|\,\alpha\in\widetilde{\Delta}_{\min,\Galg}\}.
\label{eq:defGalg}
\end{equation}
Then Bonnaf\'e proves in~\cite[(3.10), Cor. 3.12, Prop.
3.13]{bonnafequasi}

\begin{theorem}
\

\begin{enumerate}
\item For $\alpha\in\widetilde{\Delta}_{\min}$ and
$\lambda=(\lambda_{\beta})_{\beta\in\widetilde{\Delta}}\in\al$, we
have $f_{\alpha}(\lambda)_{\beta}=\lambda_{z_\alpha^{-1}(\beta)}$ for
$\beta\in\widetilde{\Delta}$. 
\item The points $\lambda,\,\mu\in\al$ are in the same
$\overline{W}_a$-orbit if and only if there is $z\in\cal A_{\Galg}$
such that $z(\lambda)-\mu\in Y(\Talg)$.
\item Let $[s]_{\Galg}\in s(\Galg)$ be a semisimple class with
representative $s\in\Galg$ corresponding
to a $\overline{W}_a$-orbit $[\lambda]$ (here, $\lambda$ denotes a
representative in $\al$) on $\al$. Then
$I_{\lambda}=\{\alpha\in\widetilde{\Delta}\,|\,\lambda_{\alpha}=0\}$
is a basis of the root system of $\Cen_{\Galg}(s)^{\circ}$ and the
component group $A_{\Galg}(s)=\Cen_{\Galg}(s)/\Cen_{\Galg}(s)^{\circ}$
is isomorphic to
$$A_{\Galg}(\lambda)=\{z\in\cal A_{\Galg}\,|\,\forall
\alpha\in\widetilde{\Delta},\,\lambda_{z(\alpha)}=\lambda_{\alpha}\}.$$
\end{enumerate}
\label{classdisc}
\end{theorem}

For $\alpha\in\widetilde{\Delta}_{\min}$, we write 
\begin{equation}
\label{eq:Inv}
V_{\alpha}=\{v\in V\,|\,f_{\alpha}(v)=v\}
\end{equation} 
for the affine subspace of $V$ of the invariants under $f_{\alpha}$.

\begin{lemma} Let $\alpha,\,\alpha'\in\widetilde{\Delta}_{\min}$. If
$\cyc{z_{\alpha}}=\cyc{z_{\alpha'}}$, then $V_{\alpha}=V_{\alpha'}$.
Moreover,
if $\Omega_{\alpha}$ denotes
the set of $\cyc{z_{\alpha}}$-orbits on $\widetilde{\Delta}$, then 
$$\dim(V_{\alpha})=|\Omega_{\alpha}|-1.$$
In Table~\ref{tab:DimInvariant}, we explicitly give the dimension of
$V_{\alpha}$ for every irreducible extended affine Dynkin diagram. For
the notation, we use the labelling of Bourbaki~\cite[Planche
I-IX]{Bourbaki456}. In particular, note that
$z_{\alpha_i}(-\alpha_0)=\alpha_i$. 
\label{DimInvariant}
\end{lemma}
\begin{table}
$$
\renewcommand{\arraystretch}{1.6}
\begin{array}{c|c|c|c|c}
\textrm{Type}&|\cal A|&\cal A&z_{\alpha}&\dim(V_{\alpha})\\
\hline
A_n&n+1&\cyc{z_{\alpha_1}}&\mbox{\scriptsize $d|n+1$},\quad z_{\alpha_1}^d&
d-1\\
\hline
B_n&2&\cyc{z_{\alpha_1}}&z_{\alpha_1}&n-1\\
\hline
C_n&2&\cyc{z_{\alpha_n}}&z_{\alpha_n}&\lfloor \frac{n}{2}\rfloor\\
\hline
D_{2n+1}&4&\cyc{z_{\alpha_n}}&z_{\alpha_n},\,z_{\alpha_{n-1}}=
z_{\alpha_n}^{-1}&n-1\\
&&&z_{\alpha_1}=z_{\alpha_n}^2&2n-1\\
\hline
D_{2n}&4&\cyc{z_{\alpha_{n-1}}}\times
\cyc{z_{\alpha_n}}&z_{\alpha_n},\,z_{\alpha_{n-1}}&n\\
&&&z_{\alpha_1}=z_{\alpha_n}z_{\alpha_{n-1}}&2n-2\\
\hline
E_6&3&\cyc{z_{\alpha_1}}&z_{\alpha_1},\,z_{\alpha_1}^{-1}&2\\
\hline
E_7&2&\cyc{z_{\alpha_1}}&z_{\alpha_1}&4\\
\hline
E_8&1&\{1\}&1&8\\
\hline
G_2&1&\{1\}&1&2\\
\hline
F_4&1&\{1\}&1&4\\
\end{array}$$
\caption{The dimension of the invariant subspace}
\label{tab:DimInvariant}
\end{table}
\subsection{Fixed-points under a Frobenius map}\label{sec:fixed-frob}
Now, we suppose that $\Galg$ is defined over the finite field $\F_q$
(for $q$ a $p$-power) and we denote by $F:\Galg\rightarrow\Galg$ the
corresponding Frobenius map. We assume that $F$ is the composition of a
split Frobenius with a graph automorphism coming from a symmetry~$\rho$ 
of $\Delta$ (which could be trivial). We suppose that the maximal torus $\Talg$
of $\Galg$ is chosen to be $F$-stable and contained in an $F$-stable Borel
subgroup $\Balg$. Moreover, $F$ induces Frobenius maps on $\Galg_{\Sc}$ and 
$\Galg_{\ad}$ (also denoted by $F$) such that $F$ and the isogenies
$\pi_{\Sc}$ and $\pi_{\ad}$ commute.
We can define an $F$-action on $X(\Talg)$ and $Y(\Talg)$ by
setting
$$F(\chi)(t)=\chi(F(t))\quad \textrm{and}\quad
F(\gamma)(x)=F(\gamma(x)),$$
for $\chi\in
X(\Talg)$, $\gamma\in Y(\Talg)$, $t\in
\Talg$ and $x\in\overline{\F}_p^{\times}$.
The $F$-action on $Y(\Talg)$ extends naturally to 
$\Q_{p'}/\Z\otimes_{\Z} Y(\Talg)$ and
to $(\Q_{p'}\otimes_{\Z}
Y(\Talg))/Y(\Talg)$ and is compatible with the isomorphisms of
Equation~(\ref{eq:tore}).
Since the set of $F$-stable semisimple
classes of $\Galg$ is in bijection
with the set $(\Talg/W)^F$; see~\cite[3.7.2]{carter2}, 
it follows that it is in bijection with the set
$(\Q_{p'}\otimes_{\Z} Y(\Talg)/\overline{W}_a)^F$.

We put $\overline{W}_{a,q}=F^{-1}(Y(\Talg))\rtimes W$. Then the map
$$F:\overline{W}_{a,q}\rightarrow\overline{W}_a,\,yw\mapsto
F(y)F(w),$$ is an isomorphism of groups. Note that $\overline{W}_a$ is a
subgroup of $\overline{W}_{a,q}$. Furthermore, for $v\in V$ and $g\in
\overline{W}_{a,q}$,  we have 
\begin{equation}
\label{eq:permutation}
F(g)(F(v))=F(g(v)).
\end{equation}
We define $W_{a,q}=F^{-1}(Y(\Talg_{\Sc}))\rtimes W$ and
$\widehat{W}_{a,q}=F^{-1}(Y(\Talg_{\ad}))\rtimes W$.
In particular, we have $W_{a,q}\leq
\overline{W}_{a,q}\leq\widehat{W}_{a,q}$. Moreover,
Equation~(\ref{eq:permutation}) implies that the set
$$\al_q=\{v\in V\,|\,F(v)\in\al\}$$ is a fundamental region for the
$W_{a,q}$-action on $V$. We denote by $h_q:V\rightarrow V$
the homothety with origin $0$ and ratio $\frac{1}{q}$. Recall that $F$
acts on $V$ by 
$h_q^{-1}\cdot F_0$, where $F_0:V\rightarrow V$ is a linear transformation 
defined by
$F_0(\alpha)=\rho^{-1}(\alpha)$ for all $\alpha\in\Delta.$\label{defF0} 
Note that $V$ can be regarded as a euclidean space on which the
elements of $\widehat{W}_{a}$ and $F_0$ act as isometries. In the
following, we will denote by $d_0$ the associated metric.
Since $\cyc{F_0(\chi),\gamma}=\cyc{\chi, F_0(\gamma)}$, we deduce that
$F_0(\om_{\alpha}^\vee)=\om_{\rho^{-1}(\alpha)}^{\vee}$ for every
$\alpha\in\Delta$. For $\alpha\in\widetilde{\Delta}$, we set
$$\om_{\alpha,q}^\vee=F^{-1}(\om_{\alpha}^\vee)=\frac{1}{q}
\cdot\om_{\rho(\alpha)}^\vee.$$ 
Note that $n_{\rho^{-1}(\alpha)}=n_{\alpha}$. Hence,
$\widetilde{\Delta}_{\min}$ is $\rho^{-1}$-stable. For $\alpha\in
\widetilde{\Delta}_{\min}$, we define 
$$
f_{\alpha,q}=z_{\rho(\alpha)}+\om_{\alpha,q}^{\vee}.$$
\begin{lemma}
Let $\alpha\in\widetilde{\Delta}_{\min}$. Then
the following diagram is commutative.
$$
\xymatrix{
   \al\ar[rrr]^{f_{\alpha}} \ar[d]^{F^{-1}} &&&
   \al\ar[d]^{F^{-1}}\\
   \al_q \ar[rrr]^ {f_{\alpha,q}}&&&
   \al_q  
   } 
$$
\label{diag1}
\end{lemma}
\begin{proof}
For every $\alpha\in\widetilde{\Delta}_{\min}$, we set
$I_{\alpha}=\Delta\backslash\{\alpha\}$ and
$\Phi_{\alpha}=W_{I_{\alpha}}(I_{\alpha})$ the corresponding root
subsystem with basis $I_{\alpha}$. Since 
$f_{-\alpha_0}=f_{-\alpha_0,q}=1$, the lemma is trivially true for
$-\alpha_0$. Fix now $\alpha\in\Delta$ with $n_{\alpha}=1$. 
For every $x\in V$, we have 
$$F^{-1}f_{\alpha}F(x)=\frac{1}{q}\cdot w_{\rho(\alpha)}^\vee
+F_0^{-1}z_{\alpha}F_0(x).$$
Moreover, we have
$\rho(I_{\alpha})=I_{\rho_{\alpha}}$ and
$\Phi_{\rho(\alpha)}=\rho(\Phi_{\alpha})$. If $w_{\alpha}$ and
$w_{\rho(\alpha)}$ are the longest
elements of $W_{I_{\alpha}}$ and $W_{I_{\rho(\alpha)}}$,  
then $w_{\rho(\alpha)}=\rho(w_{\alpha})$. Indeed, for every $\beta\in\Delta$,
we have $s_{\rho(\beta)}=F_0^{-1} s_{\beta} F_0$
which implies that $\rho(w_{\alpha})=F_0^{-1} w_{\alpha}F_0$. Now,
if $\beta\in\Phi_{\rho(\alpha)}^+$, then
$\rho(w_{\alpha})(\beta)=F_0^{-1}(-F_0(\beta))=-\beta$, as required.
So, we have
$F_0^{-1}z_{\alpha} F_0=z_{\rho(\alpha)},$
and the result follows.
\end{proof}

Note that $W_{a,q}$ is the affine Weyl group generated by $s_{\alpha}$
for $\alpha\in\Delta$
and by the affine reflection
$s_{\alpha_0,1/q}=s_{\alpha_0}+\frac{1}{q}\alpha_0^{\vee}$. We denote by
$\cal H_q$ the collection of all hyperplanes defined by the affine
reflections of $W_{a,q}$.
Moreover, $W_a$ is a subgroup of $W_{a,q}$. It follows that $\al$ is a
union of certain transforms of $\al_q$ under $W_{a,q}$. We write $E_q$ for
the set of elements $\omega\in W_{a,q}$ such that
$\omega(\al_q)\subseteq\al$ and define
\begin{equation}
\Omega_q=\{\omega(\al_q)\,|\,\omega\in E_q\}.
\label{eq:defOmega}
\end{equation}
We now can prove

\begin{theorem}
Let $\alpha\in\widetilde{\Delta}_{\min}$. We define
$$M_{\alpha,q}=\{\omega\in E_q\,|\,
f_{\alpha}\left(\omega(\al_q)\right)=\omega(\al_q)\}.$$
Let $V_{\alpha}$ be the subspace of invariants as in
Equation~(\ref{eq:Inv}). If $V_{\alpha}$ is contained in some
hyperplanes
of $\cal H_q$, then $|M_{\alpha,q}|=0$. Otherwise, we have
$|M_{\alpha,q}|=q^{\dim V_{\alpha}}.$
\label{alovefixe}
\end{theorem}

\begin{proof}
Since $\widehat{W}_a$ is a subgroup of $\widehat{W}_{a,q}$, it follows
that $f_{\alpha}\in\widehat{W}_{a,q}$. In particular, $f_{\alpha}$
permutes the elements of $\cal H_q$ and also the set of alcoves
$\omega(\al_q)$ for $\omega\in W_{a,q}$. Hence, $f_{\alpha}$ permutes
the elements of $\Omega_q$ (because for $\omega\in E_q$, we have
$f_{\alpha}(\omega(\al_q))\subseteq\al$).  
We denote by $r_{\alpha}\in W_{a,q}$ the element such that
\begin{equation}
\label{eq:ralpha}
r_{\alpha}(\al_q)=f_{\alpha}(\al_q).
\end{equation}
Let $\omega\in E_q$. Then we have
\begin{eqnarray*}
\omega(\al_q)&=&f_{\alpha}(\omega(\al_q))\\
&=&f_{\alpha}\omega f_{\alpha}^{-1}\left(f_{\alpha}(\al_q)\right)\\
&=&f_{\alpha}\omega f_{\alpha}^{-1} r_{\alpha}(\al_q),
\end{eqnarray*}
Furthermore, the group $\widehat{W}_{a,q}$ normalizes $W_{a,q}$, which
implies that $f_{\alpha}\omega f_{\alpha}^{-1}\in W_{a,q}$. In
particular, we
have $f_{\alpha}\omega f_{\alpha}^{-1}r_{\alpha}\in W_{a,q}$. 
However, by~\cite[VI.\S2.1]{Bourbaki456}, the group 
$W_{a,q}$ acts simply-transitively on the set of alcoves
$\{\omega(\al_q)\,|\,\omega\in W_{a,q}\}$. It follows that
$f_{\alpha}\omega f_{\alpha}^{-1}r_{\alpha}=\omega$, which implies
\begin{equation}
\label{eq:com}
\omega^{-1}f_{\alpha}\omega =r_{\alpha}^{-1}f_{\alpha}
\end{equation}
Note that $r_{\alpha}^{-1}f_{\alpha}(\al_q)=\al_q$ and
$r_{\alpha}^{-1}f_{\alpha}\in \widehat{W}_{a,q}$.
Proposition~\ref{DescInvariantAlcove} implies that there is
$\widetilde{\alpha}\in\widetilde{\Delta}_{\min}$ such that 
$f_{\widetilde{\alpha},q}=r_{\alpha}^{-1}f_{\alpha}$.
We define
\begin{equation}
\label{eq:alphatilde}
m:\widetilde{\Delta}_{\min}\rightarrow\widetilde{\Delta}_{\min},\,\alpha\mapsto\widetilde{\alpha}.
\end{equation}
Hence, the following diagram is commutative
$$
\xymatrix{
   \al_q \ar[rrr]^{\omega} \ar[d]^{f_{\widetilde{\alpha},q}} &&&
   \omega(\al_q)\ar[d]^{f_{\alpha}}\\
   \al_q \ar[rrr]^ {\omega}&&&
   \omega(\al_q)  
   } 
$$
For $A\subseteq V$ and $f:V\rightarrow V$, we denote by $A^f$ the subset
of elements of $A$ invariant under $f$.
So, Equation~(\ref{eq:com}) implies that
$\omega:\al_q^{f_{\widetilde{\alpha},q}}\rightarrow
\omega(\al_q)^{f_{\alpha}}$ is bijective. Thus the sets
$\al_q^{f_{\widetilde{\alpha},q}}$ and $\omega(\al_q)^{f_{\alpha}}$ are
isometric.
Let $x\in V_{\alpha}$. Suppose that $x$ lies in the interior of some
$\omega(\al_q)$ for some $\omega\in E_q$. Then $x=f_{\alpha}(x)$ lies in
the interior of $f(\omega(\al_q))$, which implies that
$f_{\alpha}(\omega(\al_q))=\omega(\al_q)$. Conversely, if
$\omega(\al_q)$ is $f_{\alpha}$-invariant, then the interior of
$\omega(\al_q)$ contains elements of $V_{\alpha}$ (for example, the
isobarycentre of the simplex $\omega(\al_q)$, because $f_{\alpha}$ is an
affine map). It follows that
$M_{\alpha,q}=\emptyset$ if and only if $V_{\alpha}$ is contained in some
hyperplane of $\cal H_{q}$.

Suppose that $V_{\alpha}$ is not contained in some hyperplane of $\cal
H_q$. For $H\in\cal H_q$, consider the affine space $V_{\alpha}\cap H$.
Then $\dim(V_{\alpha}\cap H)<\dim V_{\alpha}$ (otherwise,
$V_{\alpha}$ would be contained in $H$). This implies that
\begin{equation}
\label{eq:vol1}
\vol(\al^{f_{\alpha}})=
\vol\left(\al^{f_{\alpha}}\backslash \bigcup_{H\in
H_q,\,H\cap\al\neq\emptyset}H\right).
\end{equation}
Furthermore, we have
\begin{equation}
\label{eq:vol2}
\al^{f_{\alpha}}\backslash \bigcup_{H\in
H_q,\,H\cap\al\neq\emptyset}H=\bigsqcup_{\omega\in
M_{\alpha,q}}\overset{\circ}{\omega(\al_q)^{f_{\alpha}}}.
\end{equation}
Since $\vol(\omega(\al_q)^{f_{\alpha}})=\vol(\overset{\circ}
{\omega(\al_q)^{f_{\alpha}}})$ for every $\omega\in M_{\alpha,q}$, 
Equations~(\ref{eq:vol1}) and~(\ref{eq:vol2}) imply that
\begin{equation}
\vol(\al^{f_{\alpha}})=\sum_{\omega\in
M_{\alpha,q}}\vol\left(\omega(\al_q)^{f_{\alpha}}\right).
\label{eq:vol3}
\end{equation}
Moreover, the sets $\al_q^{f_{\widetilde{\alpha},q}}$ and
$\omega(\al_q)^{f_{\alpha}}$ are isometric. Thus they have the same
volume, and Equation~(\ref{eq:vol3}) implies that
\begin{equation}
|M_{\alpha,q}|=\frac{\vol\left(\al^{f_{\alpha}}\right)}
{\vol\left(\al_q^{f_{\widetilde{\alpha},q}}\right)}.
\label{eq:vol4}
\end{equation}
Thanks to Lemma~\ref{diag1}, we have
$F(\al_q^{f_{\widetilde{\alpha},q}})=\al^{f_{\widetilde{\alpha}}}$.
By Equation~(\ref{eq:com}) and Lemma~\ref{diag1}, $f_{\alpha}$ and
$f_{\widetilde{\alpha}}$ are conjugate (in the group of affine
transformations of $V$). 
Hence, $z_{\alpha}$ and
$z_{\widetilde{\alpha}}$ have the same order. 
If $\Galg$ is not of type $D_{2n}$, then
$\cal A$ is cyclic. This implies that
$\cyc{z_{\alpha}}=\cyc{z_{\widetilde{\alpha}}}$ and thanks to Lemma~\ref{DimInvariant}, we have
$V_{\alpha}=V_{\widetilde{\alpha}}$. Hence, we have 
$\al^{f_{\alpha}}=\al^{f_{\widetilde{\alpha}}}$.
If $\Galg$ is of type $D_{2n}$,
then the invariant subspaces of $f_{\alpha}$ and $f_{\widetilde{\alpha}}$
have the same dimension (because $f_{\alpha}$ and
$f_{\widetilde{\alpha}}$ are conjugate). Table~\ref{tab:DimInvariant}
implies that there is $i\geq 0$ such that
$\widetilde{\alpha}=\rho^i(\alpha)$.
Then $f_{\widetilde{\alpha}}=F_0^{-i}f_{\alpha}F_0^i$, and
$\al^{f_{\widetilde{\alpha}}}$ and $\al^{f_\alpha}$ are isometric.
We have proven that
$$\vol\left(F(\al_q^{f_{\widetilde{\alpha},q}})\right)=\vol\left(\al^{f_{\alpha}}\right).$$
Since $F=h_q^{-1} F_0$, we deduce that
$\vol\left(F(\al_q^{f_{\widetilde{\alpha},q}})\right)=q^{\dim(V_{\alpha})}\vol\left(\al_q^{f_{\widetilde{\alpha},q}}\right)$,
and it follows that
$\vol\left(\al_q^{f_{\widetilde{\alpha},q}}\right)=\frac{1}{q^{\dim(V_{\alpha})}}\vol(\al^{f_{\alpha}})$. 
Now, the result follows from Equation~(\ref{eq:vol4}).
\end{proof}

\begin{remark}
If $V_{\alpha}$ is contained in some hyperplane $H$ of $\cal H_q$,
then $H$ is not a wall of $\al$. Indeed, the walls of
$\al$ are the hyperplanes $H_{\beta}=\{\lambda\in
V\,|\,\cyc{\beta,\lambda}=m_{\beta}\}$ for $\beta\in\widetilde{\Delta}$,
where $m_{\beta}=0$ for $\beta\in \Delta$ and $m_{-\alpha_0}=-1$.
Let $\beta\in\widetilde{\Delta}$. We then define the element
$\lambda^{\alpha}$ by setting
$\lambda_{z_{\alpha}^i(\beta)}^{\alpha}=\frac{1}{|\cyc{
z_{\alpha}}\cdot \beta|}$ for all $i\geq 0$
 and $\lambda_{\gamma}^{\alpha}=0$ for
$\gamma\notin \cyc{z_{\alpha}}\cdot\beta$.
Note that $\lambda^{\alpha}\in\al$ and
$f_{\alpha}(\lambda^{\alpha})=\lambda^{\alpha}$. Moreover,
$\cyc{\beta,\lambda^{\alpha}}=m_{\beta}+\frac{1}{n_{\beta}|\cyc{
z_{\alpha}}\cdot \beta|}\neq m_{\beta}$, which implies that
$\lambda^{\alpha}\notin
H_{\beta}$. 
\label{rk:hyperplan}
\end{remark}

\begin{lemma}
Let $\lambda\in\omega(\al_q)$ for some $\omega\in E_q$. We suppose that
$\lambda\notin\widetilde{\Delta}_{\min}$.
Then, $\lambda$ lies in an
$F$-stable $\overline{W}_a$-orbit if and only if there is
$\alpha\in\widetilde{\Delta}_{\min}$ with $z_{\alpha}\in A_{\Galg}$ 
such that $F(\lambda)=F(\omega)f_{\alpha}(\lambda)$.
\label{elementFstable}
\end{lemma}
\begin{proof}
Since $F(\omega(\al_q))$ is an alcove for $W_a$, there is a unique $v\in
W_a$ such that $vF(\omega(\al_q))=\al$. In particular, $vF(\lambda)\in
\al$.
If $F(\lambda)$ and $\lambda$ lie in the same
$\overline{W}_a$-orbit, then $vF(\lambda)$ and $\lambda$ too.
Then, by Theorem~\ref{classdisc}(2), there is
$z_{\alpha}\in\cal A_{\Galg}$ (for $\alpha\in\widetilde{\Delta}_{\min}$)
such that $vF(\lambda)-z_{\alpha}(\lambda)\in Y(\Talg)$. Since
$\om_{\alpha}^{\vee}\in Y(\Talg)$, we deduce that
$vF(\lambda)-f_{\alpha}(\lambda)\in Y(\Talg)$. But $vF(\lambda)$ and
$f_{\alpha}(\lambda)$ lie in $\al$. We denote by
$(\mu_{\beta})_{\beta\in\widetilde{\Delta}}$ and
$(\lambda'_{\beta})_{\beta\in\widetilde{\Delta}}$ the affine coordinate
of $vF(\lambda)$ and $f_{\alpha}(\lambda)$. 
We have
$$vF(\lambda)-f_{\alpha}(\lambda)=\sum_{\beta\in\Delta}\frac{\mu_{\beta}-
\lambda'_{\beta}}{n_{\beta}}\,\om_{\beta}^{\vee}.$$
We have $0\leq \mu_{\beta}\leq 1$ and $0\leq \lambda'_{\beta}\leq 1$
for all $\beta\in\widetilde{\Delta}$ (because $vF(\lambda)$ and
$f_{\alpha}(\lambda)$ lie in $\al$).
Thus, for every $\beta\in\Delta$, we have
$$-\frac{1}{n_{\beta}}\leq
\frac{\mu_{\beta}-\lambda'_{\beta}}{n_{\beta}}\leq
\frac{1}{n_{\beta}}$$
Since $Y(\Talg)\leq Y(\Talg_{\ad})$, we deduce from Equation~(\ref{eq:yad}) that
$\mu_{\beta}=\lambda'_{\beta}$ if
$\beta\notin\widetilde{\Delta}_{\min}$.
Suppose that $\beta\in\widetilde{\Delta}_{\min}$. Then $n_{\beta}=1$ and
$\mu_{\beta}-\lambda'_{\beta}\in\{-1,0,1\}$.
If $\mu_{\beta}-\lambda_{\beta}'\neq 0$, then we can assume that
$\mu_{\beta}=1$ and $\lambda'_{\beta}= 0$. 
Then $\mu_{\beta'}=0$ for all $\beta'\neq \beta$, and
$vF(\lambda)=\om_{\beta}^{\vee}$. Similarly, we deduce that 
$f_{\alpha}(\lambda)=\om_{\beta'}^{\vee}$ for some
$\beta'\in\widetilde{\Delta}_{\min}$ (because $f_{\alpha}(\lambda)\neq
vF(\lambda)$, which implies that there is
$\beta'\in\widetilde{\Delta}_{\min}$ with
$\lambda'_{\beta'}=1$). It follows that
$\lambda=\om_{z_{\alpha}(\beta')}^{\vee}\in\widetilde{\Delta}_{\min}$.

So, if we assume that
$\lambda\notin\widetilde{\Delta}_{\min}$, we deduce that 
 $vF(\lambda)=f_{\alpha}(\lambda)$.
Moreover, $vF(\omega(\al_q))=\al$ implies that
$F(\omega)(\al)=v^{-1}(\al)$.
Since $F(\omega)\in W_a$, 
it follows from~\cite[VI.\S2.1]{Bourbaki456} that
$F(\omega)=v^{-1}$.
Conversely, if $F(\lambda)=F(\omega)f_{\alpha}(\lambda)$, 
then $F(\omega^{-1})F(\lambda)$ and
$f_{\alpha}(\lambda)$ are in $\al$. So, we have
$$F(\omega^{-1})F(\lambda)-z_\alpha(\lambda)=-\om_{\alpha}^\vee\in
Y(\Talg),$$ and the result comes from Theorem~\ref{classdisc}(2).
\end{proof}

\begin{lemma}
Suppose that $\cal A$ is cyclic.
For $\alpha\in\widetilde{\Delta}_{\min}$, $\alpha\neq -\alpha_0$ 
and $\omega\in M_{\alpha,q}$,
if $\lambda\in\omega(\al_q)$ lies in an $F$-stable
$\overline{W}_a$-orbit, then $\lambda\in V_{\alpha}$.
\label{pointfixeinvariant}
\end{lemma}

\begin{proof}
Suppose there is $\om_{\beta}^{\vee}\in\omega(\al_q)$ with
$\beta\in\widetilde{\Delta}_{\min}$. Then
$f_{\alpha}(\om_{\beta}^{\vee})=\om_{\beta}^{\vee}$. Moreover,
$\om_{\beta}^{\vee}=f_{\beta}(0)$ and
$f_{\alpha}(0)=\om_{\alpha}^{\vee}$. It follows that
$$f_{\beta}(\om_{\alpha}^{\vee})=f_{\beta}f_{\alpha}(0)=
f_{\alpha}f_{\beta}(0)=f_{\alpha}(\om_{\beta}^{\vee})=\om_{\beta}^{\vee}=f_{\beta}(0),$$
which implies that $\om_{\alpha}^{\vee}=0$, i.e. $\alpha=-\alpha_0$.

So, this proves that $\lambda\notin\widetilde{\Delta}_{\min}$ and
by Lemma~\ref{elementFstable}, there is
$\beta\in\widetilde{\Delta}_{\min}$
such that $\omega F^{-1}f_{\beta}(\lambda)=\lambda$.
Then $\lambda$ is a fixed-point of the map 
$\omega F^{-1}f_{\beta}:\al\rightarrow
\omega(\al_q)$, which is a contraction map with rapport $1/q$ with
respect to the metric $d_0$. By the contraction mapping theorem, $\lambda$
is the unique fixed-point of this map. We write
$\widetilde{\alpha}=m(\alpha)$ where $m$ is the map defined in 
Equation~(\ref{eq:alphatilde}).
Then, using Lemma~\ref{diag1} and Equation~(\ref{eq:com}),
we deduce that
\begin{eqnarray*}
\omega F^{-1}f_{\beta}f_{\widetilde{\alpha}}&=&
\omega F^{-1}f_{\widetilde{\alpha}} f_{\beta}\\
&=&\omega f_{\widetilde{\alpha},q} F^{-1}f_{\beta}\\
&=&f_{\alpha}\omega F^{-1} f_{\beta}.
\end{eqnarray*}
In particular, we have $\omega
F^{-1}f_{\beta}(\al^{f_{\widetilde{\alpha}}})=\omega(\al_q)^{f_{\alpha}}$.
Moreover, in the proof of Theorem~\ref{alovefixe} we have seen that
$\cal A$ cyclic implies that
$\al^{f_{\widetilde{\alpha}}}=\al^{f_\alpha}$. Hence, it follows that
$$\omega F^{-1}f_{\beta}(\al^{f_{\alpha}})\subseteq
\al^{f_{\alpha}}.$$
So, $\lambda$ is the limit of the sequence defined by $x_0\in
\al^{f_{\alpha}}$ and $x_k=\omega F^{-1}f_{\beta}(x_{k-1})$ for $k\geq
1$. Since $\al^{f_{\alpha}}$ is closed, we deduce that
 $\lambda\in\al^{f_{\alpha}}$, 
as required.
\end{proof}

The proof of Lemma~\ref{pointfixeinvariant} shows that the map $\omega F^{-1}f_{\alpha}:\al\rightarrow\omega(\al_q)$
for $\omega\in E_q$ and $\alpha\in\widetilde{\Delta}_{\min}$ 
 has a
unique fixed-point, denoted by $\lambda_{\omega,\alpha}$ in the
following. Moreover, we define
\begin{equation}
\label{eq:pointfixe}
S_{q,\alpha}=\{\lambda_{\omega,\alpha}\,|\,\omega\in E_{q}\}.
\end{equation}
To simplify, we will write $S_{q,-\alpha_0}=S_q$ and
$\lambda_{\omega,-\alpha_0}=\lambda_{\omega}$.

\begin{lemma}
Let $\alpha\in\widetilde{\Delta}_{\min}$.  If
$m:\widetilde{\Delta}_{\min}\rightarrow\widetilde{\Delta}_{\min}$
denotes the map defined in Equation~(\ref{eq:alphatilde}), then
$\om_{\alpha}^{\vee}\in S_q$ if and only if $m(\alpha)=\alpha$.
Moreover, if $\om_{\alpha}^{\vee}\in S_q$, then 
for $\omega\in E_q$ and $\beta\in
\widetilde{\Delta}_{\min}$ we have
$$f_{\alpha}(\lambda_{\omega,\beta})=\lambda_{\omega',\beta},$$
where $\omega'=f_{\alpha}\omega f_{\alpha}^{-1}r_{\alpha}$ and
$r_{\alpha}$ is defined in Equation~(\ref{eq:com}).
\label{lemme:img0}
\end{lemma}

\begin{proof}
We have
\begin{eqnarray*}
r_{\alpha}F^{-1}(\om_{\alpha}^{\vee})&=&r_{\alpha}F^{-1}f_{\alpha}(0)\\
&=&r_{\alpha}f_{m(\alpha),q}F^{-1}(0)\\
&=&f_{m(\alpha)}(0)\\
&=&\om_{m(\alpha)}^{\vee}.
\end{eqnarray*}
Hence, $m(\alpha)=\alpha$ if and only if $\om_{\alpha}^{\vee}\in S_q$.
For $\omega\in E_q$, we denote by $\omega'\in E_q$ the element such that
$f_{\alpha}(\omega(\al_q))=\omega(\al_q)$. We have
$$f_{\alpha}\omega
f_{\alpha}^{-1}r_{\alpha}(\al_q)=\omega'(\al_q).$$
Since $f_{\alpha}\omega f_{\alpha}^{-1}r_{\alpha}$ and $\omega'$ lie in
$W_{a,q}$, we deduce from~\cite[VI.\S2.1]{Bourbaki456} that 
$w'=f_{\alpha}\omega f_{\alpha}^{-1}r_{\alpha}$.
Now, if $m(\alpha)=\alpha$, for $\beta\in\widetilde{\Delta}_{\min}$, we
deduce that 
\begin{eqnarray*}
f_{\alpha}(\lambda_{\omega,\beta})&=&f_{\alpha}\omega F^{-1}
f_{\beta}(\lambda_{\omega,\beta})\\
&=&f_{\alpha}\omega f_{\alpha}^{-1}r_{\alpha}f_{m(\alpha),q} F^{-1}
f_{\beta}(\lambda_{\omega,\beta})\\
&=&w'F^{-1}f_{\beta}\left(f_{m(\alpha)}(\lambda_{\omega,\beta})\right)\\
&=&w'F^{-1}f_{\beta}\left(f_{\alpha}(\lambda_{\omega,\beta})\right).
\end{eqnarray*}
By unicity, we deduce that
$f_{\alpha}(\lambda_{\omega,\beta})=\lambda_{\omega',\beta}$.
\end{proof}
\begin{remark}
In~\cite[3.8]{carter2}, it is proved that $S_q\subset
\Q_{p'}\otimes_{\Z}
Y(\Talg_{\Sc})$. Since $Y(\Talg_{\Sc})\leq
Y(\Talg)$, the $\overline{W}_a$-orbits on $S_q$
parametrize some $F$-stable semisimple classes of $\Galg$.
\label{rk:classsimply}
\end{remark}

\begin{proposition}
\label{prop:premier}
Suppose that $\cal A_{\Galg}=\cyc{z_{\alpha}}$ (for some
$\alpha\in\widetilde{\Delta}_{\min,\Galg}$) has prime order and assume
that $\Galg$ is not of type $D_{2n}$. If
$V_{\alpha}$ is not contained in some hyperplane of $\cal H_q$, then
the number of $F$-stable semisimple classes of $\Galg$ with disconnected
centralizer is $q^{\dim(V_{\alpha})}$. Otherwise, there are no
$F$-stable classes with disconnected centralizer.
\end{proposition}

\begin{proof}
Since $\cal A_{\Galg}$ has prime order, the $\overline{W}_a$-orbit of
$\lambda\in\al$ corresponding to (see Equation~(\ref{eq:isoclass}))
a semisimple class of $\Galg$ with
disconnected centralizer lies in $V_{\alpha}$, by
Theorem~\ref{classdisc}(3).
First suppose that $V_{\alpha}$ is not contained in some hyperplane of
$\cal H_q$. Then, $M_{\alpha,q}$ is non-empty.
In the proof of Lemma~\ref{pointfixeinvariant}, we showed that if
$\lambda\in\omega(\al_{q})$ with
$\omega\in M_{\alpha,q}$ is such that its $\overline{W}_{a}$-orbit 
is $F$-stable, then
$\lambda=\lambda_{\omega,\beta}$ for some
$\beta\in\widetilde{\Delta}_{\min,\Galg}$. Moreover, by
Lemma~\ref{pointfixeinvariant}, we have $\lambda\in V_{\alpha}$.
Since $\cal A_{\Galg}$ is generated by $z_{\alpha}$, 
we have $V_{\alpha}\subseteq
V_{\beta}$. This implies that $f_{\beta}(\lambda)=\lambda$. Thus,
$\lambda=\lambda_{\omega,\beta}$ for all
$\beta\in\widetilde{\Delta}_{\min,\Galg}$ and $\lambda$ is the
unique element of $\omega(\al_q)$ lying in an $F$-stable $\overline W_a$-orbit.
So, in particular, $\lambda=\lambda_{\omega}$. Moreover,
Theorem~\ref{classdisc}(2) implies that
for $\omega'\in M_{\alpha,q}$ such that $\omega\neq\omega'$, the
elements $\lambda_{\omega}$ and $\lambda_{\omega'}$ are not
$\overline{W}_a$-conjugate. Then
there are $|M_{\alpha,q}|$ such elements. By
Remark~\ref{rk:classsimply}, these points lie in $\Q_{p'}\otimes_{\Z}
Y(\Talg)\cap\al$ and correspond to the $F$-stable semisimple classes of
$\Galg$ with disconnected centralizer. The claim then follows from
Theorem~\ref{alovefixe}. 

Suppose now that $V_{\alpha}$ is contained in some hyperplane of $\cal
H_q$, and assume there is $\lambda\in V_{\alpha}$ with $F$-stable
$\overline{W}_a$-orbit. Then $\lambda\notin\widetilde{\Delta}_{\min}$
(because $\alpha\neq -\alpha_0$ and no element of
$\widetilde{\Delta}_{\min}$ is fixed by $f_{\alpha}$). Hence,
Lemma~\ref{elementFstable} implies that $F(\lambda)=F(\omega)(\lambda)$
with $\omega\in E_q$ such that $\lambda\in\omega(\al_q)$, i.e.,
$\lambda=\omega F^{-1}(\lambda)$. Then, $\lambda\in S_q$. However,
by~\cite[3.8.2]{carter2}, $\lambda$ lies in a unique element of
$\Omega_q$, which contradicts the fact that $\lambda$ lies in a
hyperplane which is not a wall of $\al$ (see Remark~\ref{rk:hyperplan}). 
\end{proof}

\begin{lemma}
If $p$ does not divide $|\cal A_{\Galg}|$, then for every
$\alpha\in\widetilde{\Delta}_{\min,q}$, the invariant subspace $V_{\alpha}$ is 
contained in no hyperplane of $\cal H_q$.
\label{conditionsurq}
\end{lemma}

\begin{proof}
As we remarked at the end of the proof of
Proposition~\ref{prop:premier}, if there is $\lambda\in S_q$ such that
$f_{\alpha}(\lambda)=\lambda$ for
$\alpha\in\widetilde{\Delta}_{\min,\Galg}$, then $V_{\alpha}$ is 
contained in no hyperplane of $\cal H_q$.  
Suppose there is $\lambda\in\al$ such that:
\begin{enumerate}
\item We have $\lambda\in\Q_{p'}\otimes_{\Z}Y(\Talg_{\Sc})\cap\al$.
\item The type 
$I_{\lambda}$ is not equal to the type of $I_{\mu}$ for every
$\mu\neq\lambda$.
\item We have $f_{\alpha}(\lambda)=\lambda$.
\end{enumerate}
Let $\lambda$ be such an element. Property~(1) implies that there is a
semisimple class $[s]_{\Galg_{\Sc}}$ in $\Galg_{\Sc}$ corresponding to
$\lambda$. Then $F(s)$ is a semisimple element whose centralizer is of
the same type as the one of $s$ (because $F$ is an isogeny). 
Let $\mu$ be the point of $\al$ which is $W_a$-conjugate to
$F(\lambda)$, i.e, $\mu$ corresponds to the class of $F(s)$ in the
identification given in Equation~(\ref{eq:isoclass}). By
Theorem~\ref{classdisc}(3), we have $I_{\mu}=I_{\lambda}$ and it
follows from Property~(2) that $\lambda=\mu$. Hence, the $W_a$-orbit of
$\lambda$ is $F$-stable. We conclude using Property~(3).

Suppose that $p$ is a prime that satisfies the condition in
Table~\ref{tab:centelement}.  For types $B_n$, $C_{2n}$, $D_{2n}$,
$E_6$ and $E_7$, the corresponding elements given in
Table~\ref{tab:centelement} satisfy Properties (1), (2), and (3). The
result follows in these cases.  Furthermore, for the type $C_{2n+1}$,
if we denotes by $\lambda_n$ the corresponding element
of Table~\ref{tab:centelement}, we have
$$F(\lambda_n)=(q-\epsilon_0)\cdot\lambda_n+\epsilon_0\cdot\lambda_n,$$
where $\epsilon_0\in\{-1,1\}$ is such that $q\equiv \epsilon_0\mod 4$. Put
$\delta=0$ if $\epsilon_0=1$ and $\delta=1$ otherwise, and define
$r_n=t_nw_0^{\delta}$, with $t_n$ the translation of vector
$(q-\epsilon_0)\cdot\lambda_n\in Y(\Talg_{\Sc})$ and $w_0$ is the
longest element of $W$. By~\cite[Ch. VI,\,\S4.8]{Bourbaki456}, $w_0$ acts on $V$
as $-1$,
which implies that
$w_0(\lambda_n)=-\lambda_n$. Thus, we have $r_n\in W_a$ and
$F(\lambda_n)=r_n(\lambda_n)$. It follows that $\lambda_n\in S_q$ and
we conclude as above.

Note that if $\Galg$ is of type $A_n$, there is no element
$\lambda\in\al$ which satisfies the above properties (1), (2), (3). So,
in order to prove the result for the type $A_n$, we will more precisely
describe the hyperplanes of $\cal H_q$. For this, 
we write $\Delta=\{\alpha_i\,|\,1\leq i\leq n\}$ as 
in~\cite[Planche I]{Bourbaki456} and 
recall
that the positive roots are the elements $\alpha_{ij}=\sum_{i\leq
r<j}\alpha_r$ for
$1\leq i<j\leq n+1$. Hence, the hyperplanes of $\cal H_q$ are
given by 
$$H_{ij,k}=\left\{\lambda\in
V\,|\,\cyc{\alpha_{ij},\lambda}=\frac{k}{q}\right\},$$
for $1\leq i<j\leq n+1$ and $k\in\Z$. In particular, an element
$\lambda=(\lambda_i)_{1\leq i\leq n}$ lies in $H_{ij,k}$ if and only if
\begin{equation}
\label{eq:equahyper}
\sum_{i\leq r<j}\lambda_r=\frac{k}{q}.
\end{equation}
Let $z_{\alpha}\in\cal A$. Write $d$ for the order of $z_{\alpha}$ and
define the element $\lambda_d\in\al$ such that
$\lambda_{d,z^i_{\alpha_1}}=1/d$ for all $i\geq 0$ and
$\lambda_{d,\beta}=0$ for $\beta\in\widetilde{\Delta}$ which is not in
$\cyc{z_{\alpha}}\cdot\alpha_1$.
In particular, thanks to Equation~(\ref{eq:equahyper}), 
$\lambda_d\in H_{ij,k}$ for some $1 \leq i<j\leq n+1$ 
and $1\leq k\leq (q-1)$ (we indeed can suppose that $H_{ij,k}$ is not a
wall of $\al$ by Remark~\ref{rk:hyperplan}) if and only if 
$qn_d=kd$, where $n_d=|\{\alpha_r\,|\,i\leq
r<j\}\cap\cyc{z_{\alpha}}\cdot\alpha_1|$. Since $q>k\geq 0$, we deduce
that $p$ divides $d$.
It follows that if $p$ does not divide $|\cal A_{\Galg}|$, then $V_{\alpha}$ is
not contained in some hyperplane, as required.

Finally, if $\Galg$ is of type $D_{2n+1}$, then we show using 
equations of hyperplanes derived from~\cite[Planche III]{Bourbaki456} and an
argument similar to type $A_n$, that
the $f_{\alpha}$-stable element
$\frac{1}{4}(\om^{\vee}_1+\om^{\vee}_{2n}+\om^{\vee}_{2n+1})$ lies in
no hyperplane of $\cal H_q$ which are not walls of $\al$.
\end{proof}
\begin{table}
$$
\renewcommand{\arraystretch}{1.6}
\begin{array}{cc|c|c}
\textrm{Type}&&\lambda&\textrm{Type of }I_{\lambda}\\
\hline
B_n&\mbox{\small $p\neq 2$}&\frac{1}{2}\cdot\om_{2}^{\vee}&B_{n-2}\times A_1\times A_1\\
\hline
C_{2n}&\mbox{\small $p\neq 2$}&\frac{1}{2}\cdot\om_{n}^{\vee}&C_n\times C_n\\   
\hline
C_{2n+1}&\mbox{\small $p\neq
2$}&\frac{1}{2}\cdot\om_{2n+1}^{\vee}&A_{2n}\\   
\hline
D_{2n}&\mbox{\small $p\neq 2$}&\frac{1}{2}\cdot\om_{n}^{\vee}&D_n\times D_n\\
\hline
E_6&\mbox{\small $p\neq 3$}&\frac{1}{3}\cdot\om_{4}^{\vee}&A_2\times A_2\times A_2\\
\hline
E_7&\mbox{\small $p\neq 2$}&\frac{1}{2}\cdot\om_2^{\vee}&A_7\\
\end{array}$$
\caption{Some invariant elements}
\label{tab:centelement}
\end{table}

\begin{corollary}
In table~\ref{tab:discocent}, we give the number $n(q)$ of $F$-stable
semisimple classes with disconnected centralizer for simple algebraic
groups such that $\cal A$ has prime order. If $s$ is a representative of
such a class, we write $A_{\Galg}(s)=\Cen_{\Galg}(s)/\Cen_{\Galg}(s)^\circ$
for the component group of $\Cen_{\Galg}(s)$.
\end{corollary}
\begin{proof}
This is a direct consequence of Proposition~\ref{prop:premier} and
Table~\ref{tab:DimInvariant}. The condition on $p$ comes from
Lemma~\ref{conditionsurq}.
\end{proof}
\begin{remark}
Recall that the Lang-Steinberg theorem implies that
the number $|s(\Galg^F)|$ of semisimple classes  of the finite
fixed-point subgroup $\Galg^F$ 
is given by
\begin{equation}
|s(\Galg^F)|=\sum_{ \renewcommand{\arraystretch}{0.8}\begin{array}{c}
     \scriptstyle [s]_{\Galg}\in s(\Galg)^F\\
     \scriptstyle F(s)=s \\
     \end{array} }|A_{\Galg}(s)^F|,
\label{eq:nbclass}
\end{equation}
where the sum is over the set of $F$-stable semisimple classes of
$\Galg$ and the representative $s$ is chosen to be $F$-stable, which
is possible by the Lang-Steinberg theorem. Suppose that $\cal A_{\Galg}$
has prime order. Then every semisimple element $s$ with disconnected
centralizer has a component group $A_{\Galg}(s)$ isomorphic to $\cal
A_{\Galg}$ and so to a subgroup $H$ of $\operatorname{Z}(\Galg_{\Sc})$
of order $|\cal A_{\Galg}|$ (using the isomorphism of
Equation~(\ref{eq:isocentre})), such that the actions of $F$ on
$A_{\Galg}(s)$ and on $H$ are equivalent.  In particular, we deduce that
the actions of $F$ on the groups $A_{\Galg}(s)$ for all $s$ such that
$A_{\Galg}(s)$ is not trivial, are equivalent.  We denote by $c_1$
(resp. $c_2$) a set of representatives of the $F$-stable semisimple
classes of $\Galg$ with connected centralizer (resp. a disconnected
centralizer) and we suppose that the elements of $c_1$ and $c_2$ are
chosen to be $F$-stable. Then Equation~(\ref{eq:nbclass}) gives 
\begin{equation}
|s(\Galg^F)|=|c_1|+|H^F|\cdot|c_2|.
\label{eq:relationclass}
\end{equation}
In~\cite{Br8}, we computed $|s(\Galg^F)|$ for every simple algebraic
group $\Galg$ defined over $\F_q$; see~\cite[Table 1]{Br8}. Moreover, it
is well-known~\cite[3.7.6]{carter2} that
\begin{equation}
q^{|\Delta|}=|c_1|+|c_2|.
\label{eq:steinberg}
\end{equation}
In particular, if $|H^F|$ is not trivial (this condition is for example
related in~\cite[Table 1]{Br8}), we can deduce $n(q)=|c_2|$ from
Equations~(\ref{eq:relationclass}),~(\ref{eq:steinberg}) and~\cite[Table
1]{Br8}. We retrieve the results of Table~\ref{tab:discocent}. But, when
$H^F$ is trivial (for example, for $q\equiv 2\mod 3$ for $\Galg$ of type
$E_6$), we get in Equation~(\ref{eq:relationclass}) no new
information, and $n(q)$ cannot be computed using~\cite{Br8}. 
\label{rk:manque}
\end{remark}

\begin{remark}
In \cite{Br8}, the prime $p$ is supposed to be a good prime for $\Galg$,
i.e. $p$ does not divide any of the numbers $n_{\alpha}$ (for
$\alpha\in\Delta$) defined in Equation~(\ref{eq:heightelement}). Note
that Proposition~\ref{prop:premier} applies for any prime $p$ which not
divide $|\cal A_{\Galg}|$; see Lemma~\ref{conditionsurq}. In particular,
we deduce from Equation~(\ref{eq:nbclass}) and Table~\ref{tab:discocent}
that
$$|s\left(
{}^{\epsilon}E_{6,\ad}(2^f)\right)|=2^{6f}+2^{2f+1}\quad\textrm{and}\quad
|s\left(E_{7,\ad}(3^f)\right)|=3^{7f}+3^{4f},$$
where $\epsilon=1$ if $F$ is a split Frobenius map and $\epsilon=-1$
otherwise.
Thanks to~\cite[Proposition 5.9]{Br8}, the ordinary McKay Conjecture holds in
defining characteristic for these groups.
\label{rk:McKayE67}
\end{remark}

\begin{table}
$$
\renewcommand{\arraystretch}{1.6}
\begin{array}{ccc|c|c}
\textrm{Type}&&&n(q)&|A_{\Galg}(s)|\\
\hline
A_{n,\ad}&\mbox{\scriptsize $n+1$ prime}&\mbox{\small $p\neq
n+1$}&1&n+1\\
\hline
B_{n,\ad}&&\mbox{\small $p\neq 2$}&q^{n-1}&2\\
\hline
C_{n,\ad}&&\mbox{\small $p\neq 2$}&q^{\lfloor \frac{n}{2}\rfloor}&2\\\hline
E_{6,\ad}&&\mbox{\small $p\neq 3$}&q^2&3\\
\hline
E_{7,\ad}&&\mbox{\small $p\neq 2$}&q^4&2\\
\end{array}$$
\caption{Number of semisimple classes with disconnected centralizer}
\label{tab:discocent}
\end{table}

Now, we define
$$\Theta_{q}=\bigsqcup_{\alpha\in\
\widetilde{\Delta}_{\min,\Galg}}S_{q,\alpha}.$$

\begin{proposition}
Suppose that $F_0$ is trivial and that $q\equiv 1 \mod |A_{\Galg}|$.
Then the $\Gamma_{\Galg}$-orbits of $\Theta_{q}$ correspond to the
$F$-stable semisimple classes of $\Galg$ in the bijection given in
Equation~(\ref{eq:isoclass}).  Moreover, if $\cal A_{\Galg}$ is cyclic
and $\Galg$ is not of type $D_{2n}$, and if for
$\alpha\in\widetilde{\Delta}_{\min,\Galg}$, the invariant space
$V_{\alpha}$ is not contained in some hyperplane of $\cal H_q$, then we
have
$$|\Theta_{q,\alpha}|=q^{\dim(V_{\alpha})},$$
where  $\Theta_{q,\alpha}$ denotes the set of $\Gamma_{\Galg}$-orbits of
$\Theta_q$ whose representatives are contained in $V_{\alpha}$.
\label{prop:qmoinsun}
\end{proposition}
\begin{proof}
We have $|\cal A_{\Galg}|=|Y(\Talg)/Y(\Talg_{\Sc})|$. In particular,
$|\cal A_{\Galg}|$ is the product of the elementary divisors of
$Y(\Talg_{\Sc})\leq Y(\Talg)$ (viewed as $\Z$-modules). Hence, it
follows that $|\cal
A_{\Galg}|\cdot Y(\Talg)\leq Y(\Talg_{\Sc})$ and $|\cal A_{\Galg}|\cdot
F^{-1}(Y(\Talg))\leq F^{-1}(Y(\Talg_{\Sc})).$  
For $\alpha\in\widetilde{\Delta}_{\min}$, we have
$\om_{\alpha,q}^{\vee}=\frac{1}{q}\cdot \om_{\alpha}^{\vee}$ (because
$F_0$ is trivial). and
$\om_{\alpha}^{\vee}=(q-1)\om_{\alpha,q}^{\vee}+\om_{\alpha,q}^{\vee}$.
Denote by $r_{\alpha}$ the translation of vector
$(q-1)\om_{\alpha,q}^{\vee}$. 
Since $|\cal A_{\Galg}|$ divides $q-1$, the translation $r_{\alpha}$
lies in $F^{-1}(Y(\Talg))$, and  
we have $r_{\alpha}(\al_q)\subset\al$ and
$\om_{\alpha}^{\vee}=q\om_{\alpha,q}^{\vee}=r_{\alpha}(\om_{\alpha,q}^{\vee})$. 
Thus, $f_{\alpha}(\al_q))=r_{\alpha}(\al_q)$.
Furthermore, we have
$$r_{\alpha}F^{-1}(\om_{\alpha}^{\vee})=\om_{\alpha}^{\vee}.$$
This proves that $\om_{\alpha}^{\vee}\in S_{\alpha}$, and by
Lemma~\ref{lemme:img0}, we deduce that $m(\alpha)=\alpha$ and 
that $\Gamma_{\Galg}$ acts on $\Theta_q$.
Now, since $\om_{\alpha}^{\vee}\in \Theta_q$, Lemma~\ref{elementFstable}
implies that the elements of $\Theta_q$ are the elements of $\al$ whose
$\overline{W}_q$-orbit is $F$-stable. Moreover, by
Theorem~\ref{classdisc}(2), two elements of $\Theta_q$ are
$\overline{W}_a$-conjugate if and only if they lie in the same
$\Gamma_{\Galg}$-orbit.
As we have seen in the proof of Theorem~\ref{alovefixe},
$\Gamma_{\Galg}$ acts on the set $\Omega_q$. We denote by $S$ a
system of representatives of $\Omega_q/\Gamma_{\Galg}$.
Then we have
$$\Theta_q/\Gamma_{\Galg}=\bigsqcup_{\omega(\al_q)\in S}\,\{\Gamma_{\Galg}\cdot\lambda_{\omega,\beta}\,|\,\beta\in\widetilde{\Delta}_{\min,\Galg}\}.$$
Let $\omega(\al_q)\in S$. Then we have $|\{\Gamma_{\Galg}
\cdot\lambda_{\omega,\beta}\,|\,\beta\in\widetilde{\Delta}_{\min,\Galg}\}|=|S_{\omega,q}|$,
where
$S_{\omega,q}=\{\lambda_{\omega,\beta}\,|\,\beta\in\widetilde{\Delta}_{\min,
\Galg}\}$. Furthermore, the proof of Lemma~\ref{lemme:img0} implies that
$\lambda_{\omega,\beta}=\lambda_{\omega,\beta'}$ if
and only if there is $f\in \Gamma_{\Galg}$ with
$f(\omega(\al_q))=\omega(\al_q)$ and $\beta'$ is the element of
$\widetilde{\Delta}_{\min,\Galg}$ such that $f_{\beta'}=ff_{\alpha}$.
This proves that $$|S_{\omega,q}|=|\Gamma_{\Galg}|
/|\operatorname{Stab}_{\Gamma_{\Galg}}(\omega(\al_q))|=
|\Gamma_{\Galg}\cdot \omega(\al_q)|.$$
It follows that
$$|\Theta_q/\Gamma_{\Galg}|=\sum_{\omega(\al_q)\in S}|\Gamma_{\Galg}\cdot
\omega(\al_q)|=|E_q|=q^{|\Delta|}.$$
The last equality comes from~\cite[\S3]{DeriBrauer}.
However, the number of $F$-stable semisimple classes of $\Galg$ is
$q^{|\Delta|}$; see~\cite[3.6.7]{carter2}. 
Hence, there are at most $q^{|\Delta|}$ orbits under
$\overline{W}_a$ in $\Q_{p'}\otimes_{\Z} Y(\Talg)$ which are $F$-stable.
It follows that the elements of $\Theta_q/\Gamma_{\Galg}$ have to
contain points in $\Q_{p'}\otimes_{\Z} Y(\Talg)\cap \al$, and correspond
to the $F$-stable semisimple classes of $\Galg$.

Now, we remark that $\Gamma_{\Galg}$ stabilizes $\Theta_q\cap
V_{\alpha}$. Hence, we have $$\Theta_{q,\alpha}=\Theta_q\cap
V_{\alpha}/\Gamma_{\Galg}.$$
We set $S_{\alpha}=\{\omega(\al_q)\in S\,|\,\omega\in M_{\alpha,q}\}$.
Note that the elements of $S_{\alpha}$ do not depend on the choice of
representatives $S$ of $\Omega_q/\Gamma_{\Galg}$. Indeed, if $\omega\in
M_{\alpha,q}$ and $\omega'\in E_q$ is such that there is
$f\in\Gamma_{\Galg}$ with $\omega'(\al_q)=f(\omega(\al_q))$, then
$\omega'\in M_{\alpha,q}$ because $f$ and $f_{\alpha}$ commute.  
Moreover, for $\omega\in M_{\alpha,q}$, every $\lambda_{\omega,\beta}$
with $\beta\in\widetilde{\Delta}_{\min,\Galg}$ lies in $V_{\alpha}$
(by Lemma~\ref{pointfixeinvariant}, because $\cal A_{\Galg}$ is cyclic).
Therefore, we have
$$\Theta_{q,\alpha}=\bigsqcup_{\omega(\al_q)\in
S_{\alpha}}\{\Gamma_{\Galg}\cdot\lambda_{\omega,\beta}\,|\,\beta\in\widetilde{\Delta}_{\min,\Galg}\}.$$
As above, we conclude that $|\Theta_{q,\alpha}|=|M_{\alpha,q}|$, and the
result comes from Theorem~\ref{alovefixe}.
\end{proof}

\begin{proposition}
Suppose that $F_0$ is not trivial and that $q\equiv -1\mod |\cal
A_{\Galg}|$. Moreover, assume that $\Galg$ is not of type $D_{2n}$.
Then the conclusion of Proposition~\ref{prop:qmoinsun} holds.
\label{prop:qplusun}
\end{proposition}

\begin{proof}
Let $\alpha\in\widetilde{\Delta}_{\min}$. For every $\beta\in\Delta$, we
define
$\widetilde{s}_{\beta}=s_{\beta}-\frac{\delta_{\alpha\beta}}{q}
\cdot\alpha^{\vee}$.
Write $z_{\alpha}=\prod_{\beta\in I}s_{\beta}$ for some index subset $I$ of 
$\Delta$ and
define $$\widetilde{z}_{\alpha}=\prod_{\beta\in
I}\widetilde{s}_{\beta}\in W_{a,q}.$$ 
Note that $z_{\alpha}(\al_q)= \al_q-\frac{1}{q}\cdot\om_{\alpha}^{\vee}$
(by Proposition~\ref{DescInvariantAlcove}) and
$\widetilde{z}_{\alpha}(\al_q-\frac{1}{q}\cdot\om_{\alpha}^{\vee})=\al_q
-\frac{2}{q}\cdot\om_{\alpha}^{\vee}$.
Furthermore, we have $F_0 z_{\alpha}F_0^{-1}=z_{\rho^{-1}(\alpha)}$ (see the
proof of Lemma~\ref{diag1}). Moreover, $\rho$ acts on
$\cal A$ by $x\mapsto x^{-1}$ (because $\cal A$ is cyclic;
see~\cite[Planche I-IX]{Bourbaki456}). This implies that
$f_{\rho^{-1}(\alpha)}f_{\alpha}=\operatorname{Id}$. Hence
$$F_0z_{\alpha}F_0^{-1}(\frac{1}{q}\cdot\om_{\alpha}^{\vee})=
-\frac{1}{q}\cdot\om_{\rho(\alpha)}^{\vee}+f_{\rho^{-1}(\alpha)}
f_{\alpha}(0),$$
and we deduce that
$z_{\alpha}(F_0^{-1}(\frac{1}{q}\cdot\om_{\alpha}^{\vee}))=
-\frac{1}{q}\cdot\om_{\alpha}^{\vee}$.
Note that
$\widetilde{z}_{\alpha}(-\frac{1}{q}\cdot\om_{\alpha}^{\vee})=-\frac{1}{q}\cdot\om_{\alpha}^{\vee}$.
Since $|\cal A_{\Galg}|$ divides $(q+1)$ the translation $t$ of vector
$(q+1)\cdot\frac{1}{q}\om_{\alpha}^{\vee}$ lies in $W_a$. We set
$r_{\alpha}=t\widetilde{z}_{\alpha} z_{\alpha}\in W_a$. Then
$r_{\alpha}(\al_q)\subset \al$ and $\om_{\alpha}^{\vee}$ lies in
$r_{\alpha}(\al_q)$. Thus $f_{\alpha}(\al_q)=r_{\alpha}(\al_q)$.
Moreover, we have
$$r_{\alpha}F_0^{-1}(\frac{1}{q}\cdot\om_{\alpha}^{\vee})=\om_{\alpha}^{\vee}.$$
This proves that $\om_{\alpha}^{\vee}\in S_{\alpha}$ and we conclude as
in the proof of Proposition~\ref{prop:qmoinsun}.
\end{proof}

\begin{remark}
Suppose that $\Galg$ is of type $D_{2n+1}$ and that $p$ is odd. Assume
that $\Galg$ is of adjoint type. Thanks
to Propositions~\ref{prop:qmoinsun} and~\ref{prop:qplusun}, we have
$$
\renewcommand{\arraystretch}{1.2}
\begin{array}{c|cccc}
\alpha&-\alpha_0&\alpha_1&\alpha_{2n}&\alpha_{2n+1}\\
\hline
|\Theta_{q,\alpha}|&q^{2n+1}&q^{2n-1}&q^n&q^n
\end{array}$$
Note that we have in fact
$\Theta_{q,\alpha_{2n}}=\Theta_{q,\alpha_{2n+1}}$.
We denote by $c_d(q)$ a set of representatives (chosen to be $F$-stable) 
of $F$-stable semisimple classes of
$\Galg$ whose centralizer component group has order $d$. 
Since $\Galg$ is of adjoint type, we have $\cal A_{\Galg}=\cal
A\simeq \Z_4$ and $d|4$. 
Furthermore, since $V_{\alpha_{2n}}\subseteq
V_{\alpha_1}\subseteq V_{-\alpha_0}$, we have
\begin{equation}
|c_4(q)|=q^n,\quad
|c_2(q)|=q^{2n-1}-q^n\quad\textrm{and}\quad|c_1(q)|=q^{2n+1}-q^{2n-1}.
\label{eq:centDn}
\end{equation}
Therefore, Equation~(\ref{eq:nbclass}) implies that
$$|s(\Galg^{F})|=q^{2n+1}+q^{2n-1}+2q^n,$$
and we retrieve the result of~\cite[Table 1]{Br8}. Now,
using~\cite[Prop. 4.1]{Br8} (note that $\Galg^*$ is simply-connected), 
if we denote by $\Irr_{p'}(\Galg^{*F^*})$ the
set of irreducible $p'$-characters of $\Galg^{*F^*}$, then we deduce that
$$|\Irr_{p'}(\Galg^{*F^*})|=q^{2n+1}+3q^{2n-1}+12q^n.$$
Comparing with~\cite[Prop. 5.10]{Br8}, this proves that the ordinary McKay
Conjecture holds for $\Galg^{*F^*}$ at the prime $p$.
\label{rk:McKayDn}
\end{remark}

\section{Semisimple characters}\label{section2}
\subsection{Stable semisimple and regular
characters}\label{subsectionSemi}
In this section we keep the notation of Section~\ref{sec:fixed-frob} and
suppose that the Frobenius map $F:\Galg\rightarrow\Galg$ is split (i.e
the map $F_0:V\rightarrow V$ defined on p.\pageref{defF0} is trivial).
Moreover, we assume that $p$ is a good prime for $\Galg$.  For every
$\alpha\in\Phi$, we write $\Xalg_{\alpha}$ for the corresponding
one-dimensional subgroup of $\Galg$ normalized by $\Talg$ and choose an
isomorphism $x_{\alpha}:\overline{\F}_p\rightarrow\Xalg_{\alpha}$ in
such a way that $F(x_{\alpha}(u))=x_{\alpha}(u^q)$.  Let $\rho$ be a
symmetry of the Dynkin diagram. Then we write
$\sigma:\Galg\rightarrow\Galg$ for the graph automorphism on $\Galg$
defined for all $\alpha\in\Phi$ and $u\in\overline{\F}_p$ by
$\sigma(x_{\alpha}(u))=x_{\rho(\alpha)}(\gamma_{\alpha} u)$ where
$\gamma_{\alpha}=\pm 1$ is chosen such that $\gamma_{\pm \alpha}=1$ for
all $\alpha\in\Delta$; see~\cite[12.2.3]{carter1}.
Note that $F$ and $\sigma$ commute.  We denote by $\Ualg$ the unipotent
radical of $\Balg$. Recall that $\Balg=\Ualg\rtimes\Talg$ and that
$\Ualg=\prod_{\alpha\in\Phi^+}\Xalg_{\alpha}$.  Note that the product in
the last equation is the inner product of $\Galg$. Now, we define the
normal subgroup $$\Ualg_0=\prod_{\alpha\in\Phi^+\backslash
\Delta}\Xalg_{\alpha}\subseteq\Ualg$$ and the quotient
$\Ualg_1=\Ualg/\Ualg_0$ (with canonical projection map
$\pi_{\Ualg_0}:\Ualg\rightarrow \Ualg_1$). Then we have
$\Ualg_1\simeq\prod_{\alpha\in\Delta}\Xalg_{\alpha}$ (as direct
product), and $\Ualg_0$ is $F$-stable and connected, which implies
\begin{equation}
\label{eq:U1}
\Ualg_1^F\simeq\prod_{\alpha\in\Delta} \Xalg_{\alpha}^F,
\end{equation}
(as direct product), because $\Xalg_{\alpha}$ is $F$-stable for every
$\alpha\in\Delta$. Fix $u_1\in\Ualg^F$ such that
$\pi_{\Ualg_0}(u_1)_{\alpha}\neq 1$ for all $\alpha\in\Delta$ (such an
element is regular) and recall that $A_{\Galg}(u_1)=\operatorname
Z(\Galg)$, because $p$ is a good prime for $\Galg$;
see~\cite[14.15,\,14.18]{DM}. Then by the Lang-Steinberg theorem, we can parametrize the
$\Galg^F$-classes of regular elements by
$H^1(F,\operatorname{Z}(\Galg))$~\cite[14.24]{DM}. For $z\in
H^1(F,\operatorname{Z}(\Galg))$, we denote by $\cal U_z$ the
corresponding class of $\Galg^F$.  Furthermore, a linear character
$\phi\in\Irr(\Ualg^F)$ is regular if it has $\Ualg_0^F$ in its kernel,
and if the induced character on $\Ualg_1^F$ (also denoted by the same
symbol) satisfies $\Res_{\Xalg_{\alpha}^F}^{\Ualg_1^F}(\phi)\neq
1_{\Xalg_{\alpha}^F}$ for all $\alpha\in\Delta$. By~\cite[14.28]{DM}, we also can
parametrize the $\Talg^F$-orbits of regular characters of $\Ualg^F$ by
$H^1(F,\operatorname{Z}(\Galg))$. For this, we fix a regular character
$\phi_1$ of $\Ualg^F$. Then, for every $z\in
H^1(F,\operatorname{Z}(\Galg))$, the regular character
$\phi_z=^{t_{z}}\!\!\phi_1$, where $t_z$ is an element of $\Talg$ such
that $t_{z}^{-1}F(t_z)\in z$, is a representative of the $\Talg^F$-orbit
corresponding to $z$.

Now, for $z\in H^1(F,\operatorname{Z}(\Galg))$, we define the
corresponding Gelfand-Graev character by setting
$$\Gamma_{z}=\Ind_{\Ualg^F}^{\Galg^F}(\phi_z).$$ We denote by
$D_{\Galg^F}$ the duality of Alvis-Curtis and define
$\Irr_r(\Galg^F)=\{\chi\in\Irr(\Galg^F)\,|\,\exists z\in
H^1(F,\operatorname{Z}(\Galg)),\ \cyc{\chi,\Gamma_z}\neq 0\}$ and
\begin{equation}
\Irr_{s}(\Galg^F)=\{\epsilon_{\chi}D_{\Galg^F}(\chi)\,|\,\chi\in
\Irr_r(\Galg^F)\},
\label{eq:charsemi}
\end{equation}
where $\epsilon_{\chi}$ is a sign chosen to be such that
$\epsilon_{\chi}D_{\Galg^F}(\chi)\in\Irr(\Galg^F)$.  The elements of
$\Irr_r(\Galg^F)$ (resp. of $\Irr_s(\Galg^F)$ are the so-called regular
characters (resp. semisimple characters) of $\Galg^F$.
In order to describe more precisely the sets $\Irr_s(\Galg^F)$ and
$\Irr_r(\Galg^F)$, we first introduce further notation. We choose a
$\sigma$- and $F$-stable torus $\Talg_0$ containing
$\operatorname{Z}(\Galg)$ and we consider the connected reductive group
\begin{equation}
\label{eq:Gtilde}
\widetilde{\Galg}=\Talg_0\times_{\operatorname{Z}(\Galg)}\Galg,
\end{equation}
where $\operatorname{Z}(\Galg)$ acts on $\Galg$ and on $\Talg_0$ by
translation. We extend $\sigma$ and $F$ to $\widetilde{\Galg}$. Note that
$\widetilde{\Galg}$ has connected center and the derived subgroup of
$\widetilde{\Galg}$ contains $\Galg$. Furthermore,
$\widetilde{\Talg}=\Talg_0\Talg$ is an $F$-stable maximal torus of
$\widetilde{\Galg}$ contained in the $F$-stable Borel subgroup
$\widetilde{\Balg}=\Ualg\rtimes\widetilde{\Talg}$ of
$\widetilde{\Galg}$.
Moreover, we write $(\Galg^*,F^*)$ and $(\widetilde{\Galg}^*,F^*)$ for
pairs dual to $(\Galg, F)$ and $(\widetilde{\Galg},F)$, respectively.
Then the embedding $i:\Galg\rightarrow\widetilde{\Galg}$ induces a
surjective homomorphism $i^*:\widetilde{\Galg}^*\rightarrow\Galg^*$.
Now, we write $\cal T$ and $\widetilde{\cal T}$ for a
set of $F^*$-stable representatives of $s(\Galg^{*})^{F^*}$ and
$s(\widetilde{\Galg}^{*})^{F^*}$. Note that
$s(\widetilde{\Galg})^{F^*}=s(\widetilde{\Galg}^{F^*})$ because the
center of $\widetilde{\Galg}$ is connected, and
$\widetilde{\cal T}$ is then a system of representatives of the
semisimple classes of $\widetilde{\Galg}^{F^*}$. Furthermore, for
$s\in\cal T$, the $F^*$-stable $\Galg^*$-classes of $s$ are parametrized by
the group $H^1(F^*,A_{\Galg^*}(s))$. For $a\in H^1(F^*,A_{\Galg^*}(s))$,
we denote by $s_a$ an $F^*$-stable representative of the
$F^*$-stable class corresponding to $a$. Then
the set 
\begin{equation}
\label{eq:paramclass}
\cal S=\bigsqcup_{s\in\cal T}\left\{s_a\,|\,
a\in H^1(F^*,A_{\Galg^*}(s))\right\}
\end{equation}
is a set of representatives of $s(\Galg^{*F^*})$.  Note that the
elements of $\widetilde{\cal T}$ are chosen such that, if $s\in\cal S$,
there is $\widetilde{s}\in\widetilde{\cal T}$ with
$i^*(\widetilde{s})=s$.
Now, for any semisimple element $s\in\Galg^{*F^*}$ and
$\widetilde{s}\in\widetilde{\Galg}^{*F^*}$, we denote by $\cal
E(\Galg^F,s)\subseteq\Irr(\Galg^F)$ and $\cal
E(\widetilde{\Galg}^F,\widetilde{s})\subseteq\Irr(\widetilde{\Galg}^{F})$ the
corresponding rational Lusztig series. Recall that $\cal E(\Galg^F,s)$
consists of the irreducible constituents of Deligne-Lusztig characters
$R_{\Talg_w^{*}}^{\Galg}(s)$ with $s\in\Talg_w^*$, where $\Talg_w^*$
denotes a maximal torus of $\Galg^*$ obtained by twisting $\Talg^*$ by
$w\in W$, and we have
$$\Irr(\Galg^F)=\bigsqcup_{s\in \cal S}\cal
E(\Galg^F,s)\quad\textrm{and}\quad\Irr(\widetilde{\Galg}^F)=
\bigsqcup_{\widetilde{s}\in\widetilde{\cal
T}}\cal E(\widetilde{\Galg}^F,\widetilde{s}).$$
For $\widetilde{s}\in\widetilde{\cal T}$, let
$W^\circ(\widetilde{s}) \subseteq W$ be the Weyl group
of $\Cen_{\Galg^*}^\circ(\widetilde{s})$. We define  
\begin{eqnarray}
\rho_{\widetilde{s}}&=&\displaystyle{\frac{
\label{eq:rhos}
1}{|W^\circ(\widetilde{s})|}
\sum_{w\in W^\circ(\widetilde{s})}R_{{\Talg}_w}^{{\Galg}}
(\widetilde{s}),}\\
\chi_{\widetilde{s}}&=&\displaystyle{\frac{\varepsilon_{{\Galg}} \label{eq:chis}
\varepsilon_{\Cen_{{\Galg}^*}^\circ(\widetilde{s})}}{|W^\circ(\widetilde{s})|}
\sum_{w\in W^\circ(\widetilde{s})}\varepsilon(w)R_{{\Talg}_w}^{
{\Galg}}(\widetilde{s}),}
\end{eqnarray}
where $\varepsilon$ is the sign character of $W$ and
$\varepsilon_{{\Galg}}=(-1)^{\operatorname{rk}_{\F_q}({\Galg
})}$. Here, $\operatorname{rk}_{\F_q}({\Galg})$ denotes the $\F_q$-rank of
${\Galg}$; see~\cite[8.3]{DM}.
Then we have
$\Irr_s(\widetilde{\Galg}^F)=\{\rho_{\widetilde{s}}\,|\,\widetilde{s}
\in \widetilde{\cal T}\}$ and
$\Irr_r(\widetilde{\Galg}^F)=\chi_{\widetilde{s}}\,|\,\widetilde{s}\in\widetilde{\cal
T}\}$. Let $s\in\cal T$. Write $\widetilde{s}\in\widetilde{\cal T}$ such
that $i^*(\widetilde{s})=s$ and define
\begin{equation}
\label{eq:rhoS}
\chi_s=\Res_{\Galg^F}^{\widetilde{\Galg}^F}(\chi_{\widetilde{s}})
\quad\textrm{and}\quad
\rho_s=\Res_{\Galg^F}^{\widetilde{\Galg}^F}(\rho_{\widetilde{s}}).
\end{equation}
Furthermore, for $s\in \cal S$, we recall that
there is a surjective group homomorphism~\cite[(8.4)]{BonnafeAn}
$$\hat{\omega}_s^0:H^1(F,\operatorname{Z}(\Galg))\rightarrow\Irr\left(A_{\Galg^*}(s)^{F^*}\right).$$
We now can recall the following result~\cite[Proposition 15.3,\,
Corollaire 15.14]{BonnafeAn}.

\begin{theorem}
For every $s\in\cal S$, we have
$\cyc{\Gamma_1,\chi_s}_{\Galg^F}=1$. We write $\chi_{s,1}$ for the
common constituent and put
$\rho_{s,1}=\varepsilon_{\Galg}\varepsilon_{\Cen_{\Galg^*}^{\circ}(s)}
D_{\Galg}(\chi_{s,1})$. Moreover, for $\xi\in\Irr(A_{\Galg^*}(s)^{F^*})$, we
define $$\chi_{s,\xi}=^{t_z}\!\!\chi_{s,1}\quad\textrm{and}\quad
\rho_{s,\xi}=^{t_z}\!\!\rho_{s,1},$$ where $z$ is any elements of
$H^1(F,\cal Z(\Galg))$ such that $\hat{\omega}_s^0(z)=\xi$ and
$t_z\in\Talg$ with $t_z^{-1}F(t_z)\in z$. Then 
\begin{enumerate}
\item For $z\in H^1(F,\operatorname{Z}(\Galg))$ and
$\xi\in\Irr(A_{\Galg^*}(s)^{F^*})$, the character
$\chi_{s,\xi}$ (resp. $\rho_{s,\xi}$) is
an irreducible constituent of $\Gamma_z$ (resp. of
$D_{\Galg}(\Gamma_z)$), if and only if $\xi=\hat{\omega}_s^0(z)$.
\item We have
$$\chi_s=\sum_{\xi\in\Irr(A_{\Galg^*}(s)^{F^*})}\chi_{s,\xi}\quad\textrm{and}
\quad
\rho_s=\sum_{\xi\in\Irr(A_{\Galg^*}(s)^{F^*})}\rho_{s,\xi}.$$
\item We have $\Irr_s(\Galg^F)=\{\rho_{s,\xi}\,|\,s\in\cal
S,\ \xi\in\Irr(A_{\Galg^*}(s)^{F^*})\}$ and 
$\Irr_r(\Galg^F)=\{\chi_{s,\xi}\,|\,s\in\cal
S,\,\xi\in\Irr(A_{\Galg^*}(s)^{F^*})\}$.
\end{enumerate}
\label{decrsemi}
\end{theorem}

\begin{convention}
\label{conv1}
The regular character $\phi_1\in\Irr(\Ualg^F)$ will be
chosen to be $\sigma$-stable. This choice is possible
by~\cite[3.1]{BrHi} (note that Lemma 3.1 of~\cite{BrHi} is stated for a
Frobenius map $F'$ which commutes with $F$. But the argument is still
valuable for a graph automorphism commuting with $F$). 
\end{convention}

\begin{proposition}
Assume that $\phi_1\in\Irr(\Ualg^F)$ is chosen as in
Convention~\ref{conv1}. 
For every $z\in H^1(F,\operatorname{Z}(\Galg))$, we have
$${}^{\sigma}\Gamma_z=\Gamma_{\sigma(z)}\quad\textrm{and}\quad
{}^{\sigma}D_{\Galg^F}(\Gamma_z)=D_{\Galg^F}(\Gamma_{\sigma(z)}).$$
Moreover, the operation of $\cyc{\sigma}$ on the set of constituents of
$\Gamma_1$ and of $\cyc{\sigma^{*-1}}$ on $s(\Galg^{*F^*})$ commute,
and for $s\in\cal S$, if the $\Galg^{*F^*}$-class of
$s$ is $\sigma^*$-stable, where $\sigma^*:\Galg^*\rightarrow\Galg^*$
denotes the automorphism of $\Galg^*$ obtained in dualizing $\sigma$,
then for every $z\in H^1(F,\operatorname{Z}(\Galg))$, we have
$$^{\sigma}\chi_{s,\hat{\omega}_s^0(z)}=\chi_{s,\hat{\omega}_s^0(\sigma(z))}
\quad\textrm{and}\quad
^{\sigma}\!\rho_{s,\hat{\omega}_s^0(z)}=\rho_{s,\hat{\omega}_s^0(\sigma(z))}.
$$
\label{perm}
\end{proposition}

\begin{proof}
For $s\in \cal S$, we have 
\begin{equation}
\label{eq:permseries}
^{\sigma}\cal
E(\Galg^{F},s)=\cal E(\Galg^F,\sigma^{*-1}(s)).
\end{equation}
The proof is similar to~\cite[Proposition 1.1]{Br6}
(because $F$ and $\sigma$ commute). In
particular, one has $^{\sigma}\chi_s=\chi_{\sigma^{*-1}(s)}$ and
$^{\sigma}\rho_s=\rho_{\sigma^{*-1}(s)}$.  Since $\phi_1$ is
$\sigma$-stable, it follows that $^{\sigma}\Gamma_1=\Gamma_1$. 
This implies that if $s$ is $\sigma^*$-stable, then $\chi_{s,1}$ and
$\rho_{s,1}$ are $\sigma$-stable. We conclude as in the proof
of~\cite[Theorem 3.6]{BrHi}.
\end{proof}

\begin{remark}
Note that Proposition~\ref{perm} shows that
$H^1(F,\operatorname{Z}(\Galg))^{\sigma}$ parametrizes the
$\sigma$-stable Gelfand-Graev characters of $\Galg^F$.
\label{rk:gelfandstable}
\end{remark}

\begin{lemma}
Suppose that $H^1(F,\operatorname{Z}(\Galg))$ has prime order. Then every 
$\sigma$-stable regular (resp. semisimple) character of $\Galg^F$ is a
constituent of some $\sigma$-stable Gelfand-Graev character (resp.
dual of Gelfand-Graev character) of $\Galg^F$.
\label{regsigmastable}
\end{lemma}

\begin{proof}
Let $\chi$ be a $\sigma$-stable regular character of $\Galg^F$.
Thanks to Theorem~\ref{decrsemi}(3), there is $s\in \cal S$ and
$\xi\in \Irr(A_{\Galg^*}(s)^{F^*})$ such that $\chi=\chi_{s,\xi}$ and
Proposition~\ref{perm} implies that the $\Galg^{*F^*}$-class of $s$ is
$\sigma^*$-stable. Let $z\in H^1(F,\operatorname{Z}(\Galg))$ be any
element such that $\hat{\omega}_s^0(z)=\xi$. In particular,
$\chi\in\Gamma_z$ by Theorem~\ref{decrsemi}(1). Since
$H^1(F,\operatorname{Z}(\Galg))$ has prime order, we deduce that
$\ker(\hat{\omega}_s^0)$ is is either trivial or equal to
$H^1(F,\operatorname{Z}(\Galg))$.  If
$\ker(\hat{\omega}_s^0)=H^1(F,\operatorname{Z}(\Galg))$, then
$\xi=\hat{\omega}_s^0(1)$ and $\chi\in\Gamma_1$, which is 
$\sigma$-stable with our choice in
Convention~\ref{conv1}.  Suppose now that $\ker(\hat{\omega}_s^0)$ is
trivial.  By Proposition~\ref{perm}, we also have
$^{\sigma}\chi_{s,\xi}=\chi_{s,\hat{\omega}_s^0(\sigma(z))}$. It
follows from Theorem~\ref{decrsemi}(1) that $\chi$ is $\sigma$-stable,
if and only if $\hat{\omega}_s^0(\sigma(z))=\hat{\omega}_s^0(z)$,
which is equivalent to $z^{-1}\sigma(z)\in\ker(\hat{\omega}_s^0)$.
Then $\sigma(z)=z$ and Proposition~\ref{perm} implies that
$^{\sigma}\Gamma_z=\Gamma_z$, as required.  
\end{proof}

\begin{remark}
Note that in general, if $\chi$ is a $\sigma$-stable regular character
of $\Galg^F$, then $\chi$ is not necessarily a constituent of some
$\sigma$-stable Gelfand-Graev character of $\Galg^F$.  For example,
consider a simple simply-connected group $\Galg$ of type $A_3$ defined
over $\F_q$ (with $q\equiv 1\mod 4$) and suppose that the corresponding
Frobenius map $F$ is split. We denote by $\alpha_1,\,\alpha_2$ and
$\alpha_3$ the simple roots of $\Galg$ (relative to an $F$ and
$\sigma$-stable maximal torus $\Talg$ of $\Galg$) that we label 
as in~\cite[Planche I]{Bourbaki456}.
Write $\sigma$ for the non-trivial graph automorphism of $A_3$. The
condition on $q$ implies that $F$ acts trivially on
$\operatorname{Z}(\Galg)$. Hence,
$H^1(F,\operatorname{Z}(\Galg))=\operatorname{Z}(\Galg)$. Write $z_0$
for a generator of $\operatorname{Z}(\Galg)$. With the choice of
Convention~\ref{conv1}, we write $\Gamma_i=\Gamma_{z_0^{i}}$ (with
$0\leq i\leq 3$) for the $4$ Gelfand-Graev characters of $\Galg^F$. Furthermore, we will denote by $\om_i^{\vee}$ the fundamental weight
corresponding to $\alpha_i$. Recall that $\Galg^*=\Galg_{\ad}$. Define
$\lambda=\frac{1}{2}(\om_1^{\vee}+\om_3^{\vee})$ and
write $s=\tilde{\iota}(\lambda)\in \Talg_{\ad}$. 
Then $F(s)=s$ and $\sigma(s)=s$. Moreover, $\lambda$ is
stable under $f_{z_0^2}$, but not under $f_{z_0}$. Thus, by
Theorem~\ref{classdisc}, we have
$A_{\Galg^*}(\lambda)=\cyc{z_0^2}$, and $A_{\Galg^*}(s)^{F^*}\simeq\Z_2$.
Denote by $1$ and $\eta$ the irreducible characters of
$A_{\Galg^*}(s)^{F^*}$, and by $\chi_{s,1}$ and $\chi_{s,\eta}$ the
corresponding regular characters of $\Galg^F$ as in
Theorem~\ref{decrsemi}(1). Since $\sigma$ acts as $x\rightarrow
x^{-1}$ on $\operatorname{Z}(\Galg)$, we have
$$\hat{\omega}_s^0(1)=\hat{\omega}_s^0(z_0^2)=1\quad\textrm{and}\quad\hat{\omega}_s^0
(z_0)=\hat{\omega}_s^0
(z_0^3)=\eta,$$
and Proposition~\ref{perm} implies that
$\chi_{s,1}$ and $\chi_{s,\eta}$ are $\sigma$-stable. Moreover, thanks
to Theorem~\ref{decrsemi}(1), the Gelfand-Graev characters which have
$\chi_{s,1}$ (resp. $\chi_{s,\eta}$) as constituent are $\Gamma_0$ and
$\Gamma_2$ (resp. $\Gamma_1$ and $\Gamma_3$). 
However, by Proposition~\ref{perm}, we have
$$^{\sigma}\Gamma_0=\Gamma_0,\quad
^{\sigma}\Gamma_2=\Gamma_2\quad\textrm{and}\quad
^{\sigma}\Gamma_1=\Gamma_3,$$
and $\chi_{s,\eta}$ is a $\sigma$-stable regular character of
$\Galg^F$ which is constituent of no $\sigma$-stable Gelfand-Graev
characters of $\Galg^F$, as claimed.
\end{remark}

\subsection{Disconnected reductive groups}
By Clifford theory, an irreducible character $\chi$ of $\Galg^F$
is $\sigma$-stable, if and only if it extends to the group
$\Galg^F\semi\sigma$. Moreover, if $E(\chi)$ denotes an extension of
$\chi$, then Gallagher's theorem~\cite[6.17]{Isaacs} implies that every extension of
$\chi$ is obtained by tensoring $E(\chi)$ with a linear character
of $\Galg^F\rtimes \cyc{\sigma}$ trivial on $\Galg^F$. 
So, in order to obtain information about the
sets $\Irr_r(\Galg^F)^{\sigma}$ and $\Irr_{s}(\Galg^F)^{\sigma}$, we
aim to understand  the extensions of these characters to
$\Galg^F\semi\sigma$.  For this, we will consider the group
$$\Halg=\Galg\semi \sigma.$$ We extend $F$ to a Frobenius map on
$\Halg$ by setting $F(\sigma)=\sigma$ (to simplify notation, the
extended map will also be denoted by $F$). Note that $\Halg$ is a
disconnected reductive group defined over $\F_q$ (the rational
structure is given by $F$), and $\Halg^{\circ}=\Galg$. Moreover,
$\sigma$ is a rational quasi-central element in the sense
of~\cite[1.15]{DMnonconnexe}.  Now, for $i\geq 0$, we define a scalar
product on the space of class functions on the coset
$\Galg^F\cdot\sigma^i$, by setting
$$\cyc{\chi,\chi'}_{\Galg^F\cdot\sigma^i}=
\frac{1}{|\Galg^F|}\sum_{g\in\Galg^F}\chi(g\sigma^i)
\overline{\chi'(g\sigma^i)}.$$
Recall that in~\cite[4.10]{DMnonconnexe}, Digne and Michel define a
duality involution $D_{\Galg^F,\sigma^i}$ for $i\geq 0$ on the set of
class functions defined over the coset $\Galg^F\cdot\sigma^i$, and
prove in~\cite[4.13]{DMnonconnexe} that if $\chi\in\Irr(\Halg^F)$
restricts to an irreducible character on $\Galg^F$, then the class
function $D_{\Halg^F}(\chi)$ defined for all $g\in\Galg^F$ and $i\geq
0$ by 
\begin{equation}
\label{eq:defdualite}
D_{\Halg^F}(\chi)(g\sigma^i)=
D_{\Galg^F,\sigma^i}(\chi|_{\Galg^F\cdot\sigma^i})(g\sigma^i),
\end{equation}
is (up to a sign) an irreducible character of $\Halg^F$.

We suppose that $\phi_1\in\Irr(\Ualg^F)$ is chosen as in
Convention~\ref{conv1}. In particular, $\phi_1$ is
$\sigma$-stable and linear. Thus, $\phi_1$ extends to a linear
character $E(\phi_1)$ of $\Ualg^F\semi\sigma$ by setting
\begin{equation}
\label{eq:extcanonique}
E(\phi_1)(u\sigma)=\phi_1(u)\quad\forall u\in\Ualg^F.
\end{equation}
This extension is the so-called canonical extension of $\phi_1$.
We define 
\begin{equation}
\label{eq:gelfnc}
E(\Gamma_1)=\Ind_{\Ualg^F\semi\sigma}^{\Halg^F}(E(\phi_1)).
\end{equation}
Note that, as direct consequence of Mackey's theorem~\cite[(5.6)
p.74]{Isaacs}, we  have
$$\Res_{\Galg^F}^{\Halg^F}(E(\Gamma_1))=\Gamma_1.$$ Hence, $E(\Gamma_1)$
extends $\Gamma_1$. We write
$\Gamma_{1,\sigma}=\Res_{\Galg^F\cdot\sigma}(E(\Gamma_1))$. 

Write $C_1$ for the set of irreducible constituents of $\Gamma_1$ and
for $\chi\in C_1^{\sigma}$, denote by $E(\chi)$ the constituent of
$E(\Gamma_1)$ that extends $\chi$. Define
\begin{equation}
\Psi_1=\sum_{\chi\in C_1^{\sigma}}D_{\Halg^F}(E(\chi)).
\label{eq:defdualgelfand}
\end{equation}

\begin{lemma}We suppose that Convention~\ref{conv1} holds. Then we have
$$\Res_{\Galg^F\cdot\sigma}(\Psi_1)=D_{\Galg^F,\sigma}(\Gamma_{1,\sigma}),$$
\label{gamma1}
and
$\cyc{\Psi_1,\Psi_1}_{\Halg^F}=\cyc{\Gamma_{1,\sigma},\Gamma_{1,\sigma}}_{
\Galg^F\cdot\sigma}$. In particular, $$\cyc{\Gamma_{1,\sigma},\Gamma_{1,\sigma}}_{
\Galg^F\cdot\sigma}=\left|s(\Galg^{*F^*})^{\sigma^*}\right|.$$
\end{lemma}

In~\cite{Sorlin}, Sorlin develops a theory of Gelfand-Graev characters for
disconnected groups when $\sigma$ is semisimple or unipotent. 
These characters are extensions of some $\sigma$-stable Gelfand-Graev
characters of $\Galg^F$ to $\Halg^F$; see~\cite[\S5]{Sorlin}. 
In particular, the following result is proven~\cite[8.3]{Sorlin}.
\begin{theorem}
Suppose that $\sigma$ is a unipotent or a semisimple element of
$\Halg^F$ and that 
$H^1(F,\operatorname{Z}(\Galg^{\sigma}))$ is trivial. Then $\Halg^F$ has a
unique Gelfand-Graev character $\Gamma$ and we have 
$$\cyc{\Gamma_{\sigma},\Gamma_{\sigma}}_{\Galg^F\cdot\sigma}=
|\operatorname{Z}(\Galg^{\sigma})^{\circ
F}|q^l,$$
where $l$ is the semisimple rank of $\Galg^{\sigma}$ and
$\Gamma_{\sigma}$ denotes the restriction of $\Gamma$ to the coset
$\Galg^F\cdot\sigma$.
\label{karine}
\end{theorem}

\begin{remark}
The character $E(\Gamma_1)$ defined in Equation~(\ref{eq:gelfnc}) is a
Gelfand-Graev character of $\Halg^F$ in the sense of~\cite{Sorlin},
because the linear character $E(\phi_1)$ defined in
Equation~(\ref{eq:extcanonique}) is regular~\cite[D\'efinition
4.1]{Sorlin}). Note that by~\cite[12.2.3]{carter1} the graph
automorphism $\sigma$ that we consider here always satisfies the condition
(RS) defined in~\cite[Notation 2.1]{Sorlin}. 
\label{rk:coherence}
\end{remark}

\subsection{A result of extendibility}\label{Sectionextend} 
Let $n$ be a positive integer. The map $F'=F^n$ is a Frobenius map of
$\Galg$, which gives a rational structure over $\F_{q^n}$. Note that $F$
and $\sigma$ commute with $F'$. Then restrictions of these endomorphisms
to $\Galg^{F'}$ induce automorphisms of $\Galg^{F'}$, denoted by the
same symbol in the following. Note that, viewed as an automorphism of
$\Galg^{F'}$, the automorphism $F$ has order $n$. We write
$A=\cyc{F,\sigma}$ and 
\begin{equation}
\label{eq:norm}
N_{F'/F}:\Ualg_1^{F'}\rightarrow\Ualg_1^{F},\,u\mapsto uF(u)\ldots
F^{n-1}(u),
\end{equation} 
for the norm map of $\Ualg_1$, where $\Ualg_1$ is the group defined
before Equation~(\ref{eq:U1}), and we set
$N_{F'/F}^*:\Irr(\Ualg_1^F)\rightarrow\Irr(\Ualg_1^{F'}),\,\phi\mapsto
\phi\circ N_{F'/F}$. Since $F$ and $\sigma$ commute, we have
\begin{equation}
\sigma\circ N_{F'/F}=N_{F'/F}\circ\sigma.
\label{eq:commuteNsigma}
\end{equation}

\begin{lemma}
If $\phi$ is a $\sigma$-stable character of $\Ualg_1^F$, then the character
$N_{F'/F}^*(\phi)$ is stable under $F$ and $\sigma$. 
\label{lem:sigmaFstable}
\end{lemma}

\begin{proof}
Since $\Ualg_1$ is abelian and connected, the map $N_{F'/F}$ is
surjective~\cite[\S2.4]{BrHi}, and
$N_{F'/F}^*$ is a bijection between $\Irr(\Ualg_1^F)$ and
$\Irr(\Ualg_1^{F'})^F$. Moreover, for every
$\phi\in\Irr(\Ualg_1^F)^{\sigma}$, Equation~(\ref{eq:commuteNsigma})
implies that $N_{F'/F}(\phi)$ is $\sigma$-stable, as required.
\end{proof}

\begin{remark}
If $\phi\in\Irr(\Ualg^F)$ is regular and $\sigma$-stable, then the
corresponding character of $\Ualg_1^F$ is $\sigma$-stable. Applying
Lemma~\ref{lem:sigmaFstable} to this character, we obtain a character
of $\Irr(\Ualg_1^{F'})$ stable under $F$ and $\sigma$. Denote by
$\widetilde{\phi}$ the corresponding character of $\Ualg^{F'}$ (with
$\Ualg_0^{F'}$ in its kernel). Then $\widetilde{\phi}$ is a regular
character of $\Ualg^{F'}$ stable under $F$ and $\sigma$. Thus,
$\widetilde{\phi}$ extends to $\Ualg^{F'}\semi\sigma$. Now, 
it follows from Equation~(\ref{eq:extcanonique}) that
$E(\widetilde{\phi})$ is $F$-stable.
\label{rk:st}
\end{remark}

\begin{convention}
The character $\phi_1$ of $\Ualg^{F'}$ used to parametrize the
Gelfand-Graev characters of $\Galg^{F'}$ is chosen to be $\sigma$ and
$F$-stable. This is possible by Remark~\ref{rk:st}.
\label{conv2}
\end{convention}

\begin{proposition}Assume that $\phi_1\in\Irr(\Ualg^{F'})$ is chosen as
in Convention~\ref{conv2}.
Suppose that $\sigma$ is semisimple and that the characteristic $p$ is a
good prime of $(\Galg^{\sigma})^{\circ}$.
If
$H^1(F',\operatorname{Z}(\Galg^{\sigma}))$ is trivial, then the
constituents of $\Psi_1$ are $F$-stable.
\label{extension}
\end{proposition}

\begin{proof}
Denote by $\cal U_{\sigma}$ the set of regular elements of $\Halg$
which are
$\Galg$-conjugate to an element of the coset $\Ualg\cdot\sigma$.
In~\cite[\S8]{Sorlin}, Sorlin defines a family of subsets $(\cal U_z)_{z\in
H^1(F',\operatorname{Z}(\Galg))}$ of $\cal U_{\sigma}^F$ 
which form a partition of $\cal U_{\sigma}^F$ (see~\cite[8.1]{Sorlin}). 
Furthermore, we define
$$\gamma_{u}(g)=
\left\{
\begin{array}{cl}
|\Galg^{F'}|/|\cal U_{\sigma}^F|&\textrm{if }g\in \cal U_{\sigma}^F\\
0&\textrm{otherwise}
\end{array}
\right..
$$
Remark~\ref{rk:coherence} and the proof of~\cite[Th\'eor\`eme 8.4]{Sorlin} 
imply that
\begin{equation}
\label{eq:gammau}
D_{\Galg^F,\sigma}(\Gamma_{1,\sigma})=\gamma_u,
\end{equation}
because $H^1(F',\operatorname{Z}(\Galg^{\sigma}))$ is trivial.
Recall that the irreducible characters of $\cyc\sigma$ are described as
follows. We fix a primitive $|\cyc{\sigma}|$-complex root of unity
$\sigma_0$, and
recall that the linear characters of $\cyc{\sigma}$ are the
morphisms $\varepsilon_i:\cyc{\sigma}\rightarrow\C^{\times}$ such that
$\varepsilon_i(\sigma)=\sigma_0^i$.
Let $\rho_{s,1}$ be a $\sigma$-stable constituent of
$D_{\Galg^{F'}}(\Gamma_1)$.  
Then the set $\Irr(\Halg^{F'},\rho_{s,1})$
of extensions of $\rho_{s,1}$ to $\Halg^{F'}$
consists of the characters
$$\rho_{s,1,i}=E(\rho_{s,1})\otimes\varepsilon_i\in\Irr(\Halg^{F'}),$$ for
any $i\geq 0$, where $E(\rho_{s,1})$ denotes an extension of
$\rho_{s,1}$ to $\Halg^{F'}$ (such extensions exist
by~\cite[11.22]{Isaacs}).
Now, \cite[Proposition 8.1]{Sorlin} implies that $\cal U_{\sigma}^{F'}$
is an $\Halg^{F'}$-class (because $p$ is good for
$(\Galg^{\sigma})^{\circ}$, the group
$H^1(F',\operatorname{Z}(\Galg^{\sigma}))$ is trivial and
$\sigma$ is semisimple). Hence,
by Lemma~\ref{gamma1} and Equation~(\ref{eq:gammau}), for any
$h\in\cal U_{\sigma}^{F'}$, we have
\begin{eqnarray*}
\rho_{s,1,i}(h)&=&\sigma_0^i\rho_{s,1,i}(h)\\
    &=&\sigma_0^i\cyc{\gamma_u,\rho_{s,1}}_{\Galg^{F'}\cdot\sigma}\\
    &=&\sigma_0^i\,\varepsilon_{\Galg}\,\varepsilon_{\Cen_{\Galg^*}^{\circ}(s)}.
\end{eqnarray*}
In particular, since $\sigma_0$ has order $|\cyc{\sigma}|$, we deduce
that

\begin{equation}
\label{eq:unicite}
\rho_{s,1,i}=\rho_{s,1,j}\quad\Longleftrightarrow\quad \rho_{s,1,i}(h)
=\rho_{s,1,j}(h)\quad
\textrm{for }h\in\cal U_{\sigma}^{F'}. 
\end{equation}

Suppose now that $\rho_{s,1}$ is $F$-stable. Then
${}^F\rho_{s,1,i}=\rho_{s,1,j}$ for some $j\geq 0$, because
$\Irr(\Halg^{F'},\rho_{s,1})$ is $F$-stable, and for
$h\in\cal U_{\sigma}^{F'}$, we have

\begin{eqnarray*}
\rho_{s,1,j}(h)&=&
{}^F\rho_{s,1,i}(h)\\
&=&\rho_{s,1,i}(F(h))\\
&=&\rho_{s,1,i}(h),
\end{eqnarray*}
because $F(h)\in\cal
U_{\sigma}^{F'}$.
Therefore, Equation~(\ref{eq:unicite}) implies that
$^F\rho_{s,1,i}=\rho_{s,1,i}$.
\end{proof}

\begin{remark}
In fact, in the proof of Proposition~\ref{extension} we proved that
every extension to $\Halg^{F'}$ of an $F$- and $\sigma$-stable
constituent of $D_{\Galg^F}(\Gamma_1)$ is $F$-stable.
\end{remark}

\section{Application to finite groups of type $E_6$}\label{section3}
\subsection{Preliminaries}\label{sec:preliminaire}
In this section, $\Galg$ denotes a simple simply-connected group of type
$E_6$ over $\overline{\F}_p$. We suppose that $p$ is a good prime for
$\Galg$ (i.e., $p\neq 2,3$). 
Let $\Talg$ be a maximal torus of $\Galg$ contained in a Borel subgroup
$\Balg$ of $\Galg$. We denote by $\Phi$ the root
system of $\Galg$ relative to $\Talg$, and by $\Phi^+$ and  $\Delta$ the
sets of positive roots and simple roots corresponding to $\Balg$.
For $\alpha\in\Phi$, we write $\Xalg_{\alpha}$ for the corresponding
root subgroup and choose an isomorphism
$x_{\alpha}:\overline{\F}_p\rightarrow \Xalg_{\alpha}$. Since $\Galg$ is
simple, we have
$\Galg=\cyc{x_{\alpha}(u)\,|\,\alpha\in\Phi,\,u\in\overline{\F}_p}$.
For $\alpha\in\Phi$ and $t\in\overline{\F}_p$, we set
$n_{\alpha}(t)=x_{\alpha}(t)x_{-\alpha}(-t^{-1})x_{\alpha}(t)$. Recall
that the Weyl group $W$ of $\Galg$ is generated by the coset
$n_{\alpha}(1)\cdot \Talg$ for all $\alpha\in\Phi$. Moreover,
for $\alpha\in\Phi$ one has
$\alpha^{\vee}(t)=n_{\alpha}(t)n_{\alpha}(1)^{-1}$ for all
$t\in\overline{\F}_p$.
Then, we have $\Talg=\cyc{\alpha^{\vee}(t)\,|\,\alpha\in\Phi,\,
t\in\overline{\F}_p^{\times}}$ and
$\Balg=\cyc{\Talg,\,x_{\alpha}(u),\,\alpha\in\Phi^+,\,u\in\overline{\F}_p}$; 
see~\cite[1.12.1]{sol}.

We write $\Delta=\{\alpha_1,\ldots,\alpha_6\}$ as in~\cite[Planche
V]{Bourbaki456} and denote by $\rho$ the symmetry of $\Delta$ of order
$2$. As in~\S\ref{subsectionSemi}, we define the corresponding 
graph automorphism $\sigma:\Galg\rightarrow\Galg$ 
and a split Frobenius map by setting
$F(x_{\alpha}(u))=x_{\alpha}(u^p)$ for $\alpha\in\Phi$ and
$u\in\overline{\F}_p$, which commute with $\sigma$ and defines
an $\F_p$-structure on $\Galg$. 
Note that $\Talg$ and $\Balg$ are stable under $F$ and $\sigma$.
Moreover, by~\cite[Planche V]{Bourbaki456}, we have
$$\widetilde{\iota}(\omega_{\alpha_1}^{\vee})=\alpha_1^{\vee}(\xi)
\alpha_3^{\vee}(\xi^2)\alpha_5^{\vee}(\xi)\alpha_6^{\vee}(\xi^2),$$
where $\xi\in\overline{\F}_p$ has order $3$  (such an element exists
because $p\neq 3$). Define 
\begin{equation}
\Talg_0=\{\alpha_1^{\vee}(t)
\alpha_3^{\vee}(t^2)\alpha_5^{\vee}(t)\alpha_6^{\vee}(t^2)\,|\,t\in
\overline{\F}_p^{\times}\}.
\label{eq:radE6}
\end{equation}
Then $\Talg_0$ is a subtorus of $\Talg$
which contains
$\operatorname{Z}(\Galg)=\cyc{\widetilde{\iota}(\omega_{\alpha_1}^{\vee})}$,
and is stable under $\sigma$ and $F$.
In the following, we will use this torus for the construction of
$\widetilde{\Galg}$ as in Equation~(\ref{eq:Gtilde}).

Recall that $\Galg^{\sigma}$ is a simple group of type $F_4$
(by~\cite[1.15.2]{sol}) and $\Talg^{\sigma}$ is a maximal $F$-stable
torus of $\Galg^{\sigma}$, contained in the $F$-stable Borel subgroup
$\Balg^{\sigma}$ of $\Galg^{\sigma}$ (see~\cite[4.1.4(c)]{sol}).
In particular,
$\operatorname{Z}(\Galg^{\sigma})$ is trivial (for example, by
Proposition~\ref{DescInvariantAlcove}, 
because $Y(\Talg_{\ad}^{\sigma})/Y(\Talg_{\Sc}^{\sigma})$ is
trivial). 

\begin{lemma}
With the above notation, the group
$(\widetilde{\Galg}^{\sigma})^{\circ}$ is a simple group 
of type $F_4$.
\label{GtildeE6}
\end{lemma}

\begin{proof}
The automorphism $\sigma$ of $\widetilde{\Galg}$ stabilizes
$\widetilde{\Talg}$ and $\widetilde{\Balg}$. Then it is
quasi-semisimple (see~\cite[1.1]{DMnonconnexe}) and 
by~\cite[0.1]{DMpointfixe}, the group $\widetilde{\Galg}^{\sigma}$ is
reductive and the root system of $(\widetilde{\Galg}^{\sigma})^{\circ}$
only depends on $\Phi$ (the root system of $\widetilde{\Galg}$)
and on $\sigma$. Hence, $(\widetilde{\Galg}^{\sigma})^{\circ}$ and
$\Galg^{\sigma}$ have the same type, i.e., an irreducible root system 
of type $F_4$.
Furthermore, by~\cite[1.8]{DMnonconnexe}, 
$\Talg'=(\widetilde{\Talg}^{\sigma})^{\circ}$ is a maximal
torus of $(\widetilde{\Galg}^{\sigma})^{\circ}$.
Now the exact sequence 
\begin{equation}
\label{eq:exactradical}
0\rightarrow (X(\Talg')\cap \Q\Phi)/\Z\Phi\rightarrow
X(\Talg')/\Z\Phi\rightarrow
X(\Talg')/(X(\Talg')\cap
\Q\Phi)
\end{equation}
induces an exact sequence for the $p'$-torsion subgroups of these
groups. Since $X(\Talg')/(X(\Talg')\cap\Q\Phi)$ has
no torsion, we deduce that
\begin{equation}
\left((X(\Talg')\cap \Q\Phi)/\Z\Phi\right)_{p'}\simeq 
\left(X(\Talg')/\Z\Phi\right)_{p'}
\label{eq:ptorsion}
\end{equation}
However, the group $X(\Talg')\cap\Q\Phi$ is a subgroup of the weight
lattice $\Lambda$. So, it follows that
$((X(\Talg')\cap\Q\Phi)/\Z\Phi)_{p'}$ is
a subgroup of $X(\Talg_{\Sc}^{\sigma})/X(\Talg_{\ad}^{\sigma})=\{1\}$
(because $\Galg^{\sigma}$ 
is a simple group of type $F_4$, which implies that its
fundamental group is trivial).  
It follows from Equation~(\ref{eq:ptorsion}) and~\cite[4.1]{BonnafeAn}
that $\operatorname{Z}( (\widetilde{\Galg}^{\sigma})^{\circ})$ is
connected.
Denote by $\chi_0:\widetilde{\Talg}\rightarrow\overline{\F}_p^{\times}$
the character of $\widetilde{\Talg}$ induced by the character
$\Talg_0\rightarrow \overline{\F}_p^{\times},\,t\mapsto t^3$ in $X(\Talg_0)$ 
(the character $\chi_0$
is well-defined because it is trivial on $\operatorname{Z}(\Galg)$).
Note that
$X(\widetilde{\Talg})=\cyc{\alpha,\,\alpha\in\Delta;\,\chi_0}$.
Moreover, ${}^{\sigma}\chi_0=-\chi_0$ implies that
$$\rk_{\Z}\left((1-\sigma)X(\widetilde{\Talg})\right)=\rk_{\Z}\left(
(1-\sigma)X(\Talg) \right)+1.$$ 
Now, the proof of~\cite[1.28]{DMnonconnexe} implies that
\begin{eqnarray*}
\dim(\Talg')&=&\rk_{\Z}\left(X(\widetilde{\Talg})/
(1-\sigma)X(\widetilde{\Talg})\right)\\
&=&\rk_{\Z}\left(X(\widetilde{\Talg})\right)-\rk_{\Z}\left(
(1-\sigma)X(\widetilde{\Talg})\right)\\
&=&\rk_{\Z}(X(\Talg))-\rk_{\Z}\left( (1-\sigma)X(\Talg) \right)\\
&=&\dim(\Talg^{\sigma}).
\end{eqnarray*}
Hence, if $\Phi_{\sigma}$ denotes the root system of $\Galg^{\sigma}$,
we deduce from~\cite[8.1.3]{Springer} that
\begin{eqnarray*}
\dim(
(\widetilde{\Galg}^{\sigma})^{\circ})&=&\dim(\Talg')+|\Phi_{\sigma}|\\
&=&\dim(\Talg^{\sigma})+|\Phi_{\sigma}|\\
&=&\dim( \Galg^{\sigma} ).
\end{eqnarray*}
The result follows.
\end{proof}

Let $n$ be a positive integer, $F'=F^n$ and $A=\cyc{F,\sigma}$ as
in~\S\ref{Sectionextend}. We consider the finite group $\Galg^{F'}$.
\begin{lemma}
The subgroups of $A$ are $\cyc{\sigma^i F^j}$ (for $i\in\{1,3\}$ and
a divisor $j$ of $n$) and $\cyc{\sigma}\times\cyc{F^j}$ (for a divisor
$j$ of $n$).
\label{aut}
\end{lemma}
\begin{proof}
Note that the elements of order $2$ of $A$ are $\sigma$, $F^{n/2}$ and
$\sigma F^{n/2}$ if $n$ is even and $\sigma$ otherwise.  Let $H$ be a
subgroup of $A$. Then $H=H_2\times H_{2'}$ with $H_2$ the $2$-Sylow
subgroup of $H$ and $H_{2'}\leq \cyc{F}$. Then $H$ is cyclic if and
only if $H_2$ is cyclic if and only $H_2$ contains a unique element of
order $2$. The result comes from the fact that if $H_2$ contains more
than one element of order $2$, then $\sigma\in H_2$.
\end{proof}

Denote by $\Irr_{l}(\Ualg^{F'})$ the set of linear characters of
$\Ualg^{F'}$ and by $\Irr_s(\Balg^{F})$ the set of irreducible
characters $\chi$ of $\Balg^{F'}$ such that
$\Res_{\Ualg^{F'}}^{\Balg^{F'}}(\chi)$ has constituents in
$\Irr_l(\Ualg^{F'})$. The characters of $\Irr_s(\Balg^{F'})$ can be
described as follows. The $\widetilde{\Talg}^{F'}$-orbits of
$\Irr_l(\Ualg^{F'})$ are parametrized by the subsets of $\Delta$.
For $J\subset \Delta$, we denote by $\omega_J$ the corresponding
$\widetilde{\Talg}^{F'}$-orbit, and write
$\Lalg_J$ for
the standard Levi subgroup (which is $F$-stable) with set of simple roots 
$J$. Note that $\omega_J$ corresponds to the regular characters of
$\Irr(\Ualg_J^{F'})$.
\begin{convention}
Write $A_J=\Stab_A(\omega_J)$. Then by~\cite[Lemma 3.1]{BrHi} and
Remark~\ref{rk:st} there is $\phi_J$ in $\omega_J$ an $A_J$-stable character. 
Moreover, $\Stab_{A}(\phi_J)=A_J$ and if
$\tau\in A$, then $\Stab_{A}({}^{\tau}\phi_J)=\Stab_A(\phi_J)$,
because $A$ is abelian.  Let $\Omega$ be an $A$-orbit of
$\Irr_l(\Ualg^{F'})/\widetilde{\Talg}^{F'}$ and we fix $\omega_J\in
\Omega$.  In the following, we will fix a $A_J$-stable character
$\phi_J\in\omega_J$, and if $J'\subseteq \Delta$ is such
that there is $\tau\in A$ with $\omega_{J'}={}^{\tau}\omega_J$, then
we choose $\phi_{J'}={}^{\tau}\phi_J$ as representative for
$\omega_{J'}$. Note that $\phi_{J'}$ is well-defined (because it does
not depend on the choice of $\tau\in A$ with
${}^{\tau}\omega_J=\omega_{J'}$) and is $A_{J'}$-stable.  This choice
is compatible with Convention~\ref{conv2}.
\label{conv3}
\end{convention}
Now, for $z\in H^1(F',\operatorname{Z}(\Lalg_J))$,
we choose $t_z\in \Talg$ such that $t_z^{-1}F'(t_z)\in z$ and define
\begin{equation}
\phi_{J,z}={}^{t_z}\phi_J.
\label{eq:linparam}
\end{equation}
Then the family $(\phi_{J,z})_{J\subseteq\Delta,\,z\in
H^1(F',\operatorname{Z}(\Lalg_J))}$ is a system of representatives of
the $\Talg^{F'}$-orbits of $\Irr_l(\Ualg^{F'})$.
Moreover, if we write $\Zz_J=\operatorname{Z}(\Lalg_J)$, then for
every $z\in H^1(F',\operatorname{Z}(\Lalg_J))$, we have
$\Stab_{\Talg^{F'}}(\phi_{J,z})=\Zz_J^{F'}$. Now, for $J\subseteq
\Delta$, $z\in H^1(F',\operatorname{Z}(\Lalg_J))$ and
$\psi\in\Irr(\Zz_J^{F'})$, we define

\begin{equation}
\chi_{J,z,\psi}=\Ind_{\Ualg^F\rtimes
\Zz_J^{F'}}^{\Balg^{F'}}(\hat{\phi}_{J,z}\otimes\psi),
\label{eq:defcharB}
\end{equation}
where $\hat{\phi}_{J,z}$ is the extension of $\phi_{J,z}$ to
$\Ualg^{F'}\rtimes \Zz_J^{F'}$ defined by
$\hat{\phi}_{J,z}(ut)=\phi_{J,z}(u)$ for all $u\in \Ualg^{F'}$ and
$t\in \Zz_J^{F'}$.

\begin{lemma}
Assume that Convention~\ref{conv3} holds.
%
For $\tau\in A$, $J\subseteq\Delta$, $z\in H^1(F',\Zz(\Lalg_J))$ and
$\psi\in\Irr(\Zz_J^{F'})$, we have
$${}^{\tau}\chi_{J,z,\psi}=\chi_{\tau(J),\tau(z),{}^{\tau}{\psi}}.$$
\label{imgautoB}
\end{lemma}

\begin{proof}
Using the induction formula~\cite[5.1]{Isaacs}, we have
\begin{equation}
{}^{\tau}\chi_{J,z,\psi}=\Ind_{\Ualg^{F'}\rtimes
\tau(\Zz_{J}^{F'})}^{\Balg^{F'}}
({}^\tau\hat{\phi}_{J,z}\otimes {}^\tau\psi).
\end{equation}
Since $\tau$ and $F'$ commute, we have
$\tau(\Zz_J^{F'})=\tau(\Zz_J)^{F'}=\Zz_{\tau(J)}^{F'}$, because
$\tau(\Lalg_J)=\Lalg_{\tau(J)}$. Moreover,
the choices in Convention~\ref{conv3} imply that
${}^{\tau}\phi_{J,z}$ is $\Talg^{F'}$-conjugate to 
$\phi_{\tau(J),\tau(z)}$,
and the result follows.
\end{proof}

\subsection{Equivariant bijections}
We define $B=D\rtimes A$, where $D$ is the group of outer diagonal
automorphism of $\Galg^{F'}$ induced by the inner automorphisms of
$\widetilde{\Talg}^{F'}/\Talg^{F'}$.
We denote by $\cal O_D$ and $\cal O'_D$ the set of $D$-orbits of
$\Irr_s(\Galg^{F'})$ and $\Irr_s(\Balg^{F'})$. The group $D$ has order
$1$ or $3$. For $i\in\{1,3\}$, we write $\cal O_{D,i}$ and $\cal
O'_{D,i}$ for the subset of elements of $\cal O_D$ and $\cal O'_D$ of
size $i$.  
For $\nu\in\Irr(\Zz(\Galg^{F'})$, we denote by $\Irr(\Galg^{F'}|\nu)$
the set of irreducible characters of $\Galg^{F'}$ lying over
$\nu$, that is $\chi\in\Irr(\Galg^{F'}|\nu)$ if and only if
$\chi|_{\Zz(\Galg^{F'})}=\chi(1)\cdot \nu$.
We recall that $\Zz(\Balg)=\Zz(\Galg)$, and for
$\nu\in\Irr(\Zz(\Galg^{F'}))$ we set
$$\Irr_s(\Galg^{F'}|\nu)=\Irr(\Galg^{F'}|\nu)\cap\Irr_s(\Galg^{F'})\quad\textrm{and}\quad
\Irr_s(\Balg^{F'}|\nu)=\Irr(\Balg^{F'}|\nu)\cap\Irr_s(\Balg^{F'}),$$
and denote by $\cal O_{D,\nu}$ and $\cal O'_{D,\nu}$ the set of
$D$-orbits of $\Irr(\Galg^{F'}|\nu)$ and $\Irr(\Balg^{F'}|\nu)$,
respectively. For $i\in\{1,3\}$, we write $\cal O_{D,\nu,i}$ (resp. $\cal
O'_{D,\nu,i}$) for the subset of elements of $\cal O_{D,\nu}$ (resp. of
$\cal O'_{D,\nu}$) of size $i$.

\begin{remark}
\label{rk:Dorbits}
By Theorem~\ref{decrsemi}(2), every  $D$-orbit of $\Irr_s(\Galg^{F'})$
is the set of constituents of some $\rho_s$ with $s\in\cal S$. We denote
by $\delta_s$ the $D$-orbit corresponding to $s\in \cal S$. Note that
$|\delta_s|=|A_{\Galg^*}(s)^{F'}|$. 
For $J\subseteq \Delta$ and $\psi\in\Irr(\Zz_J^{F'})$, we define
\begin{equation}
\label{eq:paramDorbitB}
\delta_{J,\psi}=\{\chi_{J,z,\psi}\,|\,z\in
H^1(F',\Zz(\Lalg_J))\}.
\end{equation}
Then the $D$-orbits of $\Irr_s(\Balg^{F'})$ are the sets
$\delta_{J,\psi}$ with $J\subseteq\Delta$ and $\psi\in\Irr(\Zz_J^{F'})$.
Moreover, we have $|\delta_{J,\psi}|=|H^1(F',\Zz(\Lalg_J))|$.
\end{remark}

\begin{lemma}Let $\nu\in\Irr(\Zz(\Galg^{F'})$. Write
$A_{\nu}=\Stab_A(\nu)$ and suppose that $\cal O_{D,\nu,k}$ and $\cal
O'_{D,\nu,k}$ for $k\in\{1,3\}$ are $A_{\nu}$-equivalent. Then 
$\Irr_s(\Galg^{F'}|\nu)$ and $\Irr_s(\Balg^{F'}|\nu)$ are
$D\rtimes A_{\nu}$-equivalent.
\label{equiv1}
\end{lemma}

\begin{proof}
We choose $A_{\nu}$-equivariant bijections $f_1:\cal
O_{D,\nu,1}\rightarrow \cal O'_{D,\nu,1}$ and $f_3:\cal
O_{D,\nu,3}\rightarrow \cal O'_{D,\nu,3}$.
We define
$\Psi_{\nu}:\Irr_s(\Galg^{F'}|\nu)\rightarrow\Irr_{s}(\Balg^{F'}|\nu)$
as follows. Let $\delta_s\in\cal O_{D,\nu}$. If $\delta_s\in\cal
O_{D,\nu,k}$ for $k\in\{1,3\}$, then by Remark~\ref{rk:Dorbits} there is
$J\subseteq \Delta$ and $\psi\in\Irr(\Zz_J^{F'})$ such that
$f_k(\delta_s)=\delta_{J,\psi}$.  Then we set
$$\Psi_{\nu}(\rho_{s,z})=\chi_{J,z,\psi}.$$ Note that, if
$H^1(F',\Zz(\Lalg_J))$ is not trivial, then $H^1(F',\Zz(\Lalg_J))$ and
$H^1(F',\Zz(\Galg))$ are identified by the map
$h_J^1:H^1(F',\Zz(\Galg))\rightarrow  H^1(F',\Zz(\Lalg_J))$ defined
in~\cite[14.31]{DM} (which is an isomorphism in this case).  Hence, the
map $\Psi_{\nu}$ is well-defined and is an $D\rtimes
A_{\nu}$-equivariant bijection by Theorem~\ref{decrsemi},
Lemma~\ref{imgautoB}, Proposition~\ref{perm} and~\cite[3.6]{BrHi}.
\end{proof}

\begin{theorem}
Suppose that $\Galg$ is a simple simply-connected group of type $E_6$
defined over $\F_q$ with corresponding Frobenius $F'$. We suppose that
$F'$ is split. 
With the above notation, if $\nu\in\Irr(\Zz(\Galg)^{F'})$, then the
sets
$\Irr_s(\Galg^{F'}|\nu)$ and $\Irr_s(\Balg^{F'}|\nu)$ are $D\rtimes
A_{\nu}$-equivalent.
\label{equive6}
\end{theorem}

\begin{proof} Write $q=p^n$, $F$ the split Frobenius map of $\Galg$
over $\F_p$
which stabilizes $\Talg$ and $\Balg$, and $\sigma$ the graph
automorphism of $\Galg$ with respect to $\Talg$ and $\Balg$, as above.
Recall that $F'=F^n$. We will prove that $\Irr_s(\Galg^{F'}|\nu)$
and $\Irr_s(\Balg^{F'}|\nu)$ are $D\rtimes A_{\nu}$-equivalent,
using Lemma~\ref{equiv1}. In order to prove that $\cal O_{D,\nu,k}$ and $\cal
O'_{D,\nu,k}$ for $k\in\{1,3\}$ are $A_{\nu}$-equivalent, we
use~\cite[13.23]{Isaacs}.

First, we suppose that $\Zz(\Galg)^{F}=1$ and
$\Zz(\Galg)^{F'}=\Zz(\Galg)$, that is $p\not\equiv 1\mod 3$ and $q\equiv
1\mod 3$. We write
$\Irr(\Zz(\Galg)^{F'})=\{1_Z,\varepsilon,\varepsilon^2\}$.
Note that $n$ is even. Moreover, we have
$$A_{1_Z}=A\quad\textrm{and}\quad
A_{\varepsilon}=A_{\varepsilon^2}=\cyc{\sigma F}.$$
Suppose that $k=3$ and $\nu=1_Z$.  
Let $H\leq A$ (the 
subgroups of $A$ are described in Lemma~\ref{aut}). 
By Lemma~\cite[5.7]{BrHi} and Equation~(\ref{eq:permseries}) (which is
valid for any element of $A$), 
we deduce that $|\cal O_{D,1_{Z},3}^H|$ is equal to the number of
$H^*$-stable classes of $s(\Galg^*)$ with disconnected centralizer, where 
$H^*$ denotes the subgroup of automorphisms of $\Galg^*$ induced by
elements of $H$. 
If $H=\cyc{\sigma}\times \cyc{F^j}$ or $H=\cyc{\sigma^iF^j}$, then we
write $d=\ord(F^j)$.
We claim that $|\cal O_{D,1_Z,3}^H|=p^{2n/d}$. Indeed,
Theorem~\ref{classdisc}(3) implies that every semisimple class of
$\Galg$ with disconnected centralizer is $\sigma$-stable, and we conclude
with Table~\ref{tab:discocent}.

Now, in the proof of~\cite[5.9]{BrHi}, it is shown that 
$|\delta'_{J,\psi}|=3$, if and only if
$\Zz_J^{F'}=\Talg_J^{F'}\times H_J^{F'}$, where
$\Talg_J$ is a torus of rank $|\Delta|-|J|$ and $H_J^{F'}$ is isomorphic to
$H^1(F',\Zz(\Lalg_J))$. Moreover, the elements of
$\delta'_{J,\psi\otimes\varepsilon^m}$
lie over $\varepsilon^m$. 
%
So, $\cal
O'_{D,\varepsilon^m,3}$ consists of the orbits
$\delta'_{J,\psi\otimes\varepsilon^m}$ with $|H^1(F',\Zz(\Lalg_J))|=3$. 
By~\cite[Lemme 2.16, Table 2.17]{BonLevi}, a subset $J\subseteq
\Delta$ parametrizes such an orbit if and only if it contains 
$\{\alpha_1,\alpha_3,\alpha_5,\alpha_6\}$. 
Furthermore, by~\cite[Lemma 4.4.7]{sol}, the group
$\Galg^{\sigma}$ is a simple group of type $F_4$ with root system
$\Phi_{\sigma}=\{\widetilde{\alpha}\,|\,\alpha\in\Phi\}$
where $\widetilde{\alpha}=\frac{1}{2}(\alpha+\rho(\alpha))$,
and the set of simple roots
of $\Galg^{\sigma}$ with respect to $\Talg^{\sigma}$ and
$\Balg^{\sigma}$ is
$\Delta_{\sigma}=\{\widetilde{\alpha}_1,\ldots,\widetilde{\alpha}_4\}$.
Note that the labelling is as in~\cite[Planche VIII]{Bourbaki456}. In particular,
$\widetilde{\alpha}_1=\alpha_2$ and $\widetilde{\alpha}_2=\alpha_4$ are
the long roots of $\Delta_{\sigma}$. Moreover, the root subgroup
 corresponding to
$\widetilde{\alpha}\in\Phi_{\sigma}$ is
$\widetilde{\Xalg}_{\widetilde{\alpha}}^{\sigma}$, where
$\widetilde{\Xalg}_{\widetilde{\alpha}}=\Xalg_{\alpha}$ if
$\alpha=\rho(\alpha)$ and 
$\widetilde{\Xalg}_{\widetilde{\alpha}}=\Xalg_{\alpha}\cdot\Xalg_{\rho(\alpha)}$ if
$\alpha\neq\rho(\alpha)$; see the proof of~\cite[Lemma 4.4.7]{sol}.
We associate to $J\subseteq\Delta$ the subset
$\widetilde{J}\subseteq\Delta_{\sigma}$ such that, if a
$\sigma$-orbit of $\Delta$ lies in $J$, the corresponding root of
$\Delta_{\sigma}$ lies in $\widetilde{J}$. Write
$\Phi_{\widetilde{J}}=\Phi_{\sigma}\cap\Z\widetilde{J}$.
Then we have
$$\Lalg_J=\cyc{\Talg,\widetilde{\Xalg}_{
\widetilde{\alpha}},\,\widetilde{\alpha}\in
\Phi_{\widetilde{J}}}\quad\textrm{and}\quad
\Lalg_{\widetilde{J}}=\cyc{\Talg^{\sigma},
\widetilde{\Xalg}^{\sigma}_{\widetilde{\alpha}},
\,\widetilde{\alpha}\in\Phi_{\widetilde{J}}},$$
where $\Lalg_{\widetilde{J}}$ is the standard Levi
subgroup of  $\Galg^{\sigma}$ with respect to $\Talg^{\sigma}$
corresponding to $\widetilde{J}$, because $\Talg^{\sigma}$ is connected.
Note that 
$\Lalg_J^{\sigma}=\Lalg_{\widetilde{J}}$ and
$D(\Lalg_J^{\sigma})=D(\Lalg_{J})^{\sigma}$, where $D(\Lalg_J)$ denotes
the derived subgroup of $\Lalg_J$.
Furthermore, we have $\Talg_J=\rad(\Lalg_J)$, which implies that $\Talg_J$
is $\sigma$-stable. 
Since $\Lalg_J^{\sigma}$ is connected
(as a Levi subgroup), we deduce from~\cite[1.31]{DMnonconnexe}
and~\cite[2.2.1]{Springer} that
$\Lalg_J^{\sigma}=\Talg_J^{\sigma}\,D(\Lalg_J)^{\sigma}$. This
product is direct, because $\sigma$ fixes no non-trivial element of the center of
$D(\Lalg_J)$. It follows that $\Talg_J^{\sigma}$ is connected (as
group isomorphic to the connected quotient
$\Lalg_J^{\sigma}/D(\Lalg_J)^{\sigma}$). In particular,
$\Talg_J^{\sigma}$ is a subgroup of $\rad(\Lalg_J^{\sigma})$.
Moreover, since $D(\Lalg_J)^{\sigma}=D(\Lalg_J^{\sigma})$, we deduce
from~\cite[2.3.3]{Springer} that 
$\dim(\Talg_J^{\sigma})
=\dim(\rad(\Lalg_J^{\sigma}))$,  
and~\cite[1.8.2]{Springer} implies that
\begin{equation}
\label{eq:rad}
\Talg_{J}^{\sigma}=\rad(\Lalg_J^{\sigma}).
\end{equation}
Let $\delta'_{J,\psi\otimes 1_Z}\in \cal O'^{\cyc{\sigma}
\times\cyc{F^j}}_{D,1_Z,3}$. Then $J$ is $\sigma$-stable and contains
$\{\alpha_1,\alpha_3,\alpha_5,\alpha_6\}$, and
${}^{\sigma}\psi=\psi$. Moreover, we have
$(\Talg_J^{F'})^{\cyc{\sigma,F^j}}=(\Talg_{J}^{\sigma})^{F^j}$ and
Equation~(\ref{eq:rad}) implies that the set $\cal O'^{\cyc{\sigma}
\times\cyc{F^j}}_{D,1_Z,3}$ is in bijection with the set of
characters
$\widetilde{\chi}_{\widetilde{J},\widetilde{\psi}}$ of 
$\Irr_s((\Balg^{\sigma})^{F^j})$ 
such that $\widetilde{J}$
contains $\{\widetilde{\alpha}_3,\widetilde{\alpha}_4\}$.
Therefore,
\cite[Lemma
5.4]{Br8} implies that 
\begin{equation}
\label{eq:int1}
\left|O'^{\cyc{\sigma}
\times\cyc{F^j}}_{D,1_Z,3}\right|=p^{2{n/d}}.
\end{equation}
Similarly, since
$(\sigma^i
F^j)^n=F'$ (because $n$ is even), for a $\sigma^i$-stable
$J\subseteq\Delta$, we have $(\Talg_J^{F'})^{\sigma^i
F^j}=\Talg_J^{\sigma^i F^j}$, and we deduce that
 $\cal O'^{\cyc{\sigma^i F^j}}_{D,1_Z,3}$
is in bijection with the $D$-orbits of size $3$ of
$\Irr_s(\Balg^{\sigma^i
F^j},1_Z)$. As above,~\cite[2.17]{BonLevi}
and~\cite[Lemma 5.4]{Br8} implies that $|\cal
O'^{\cyc{\sigma^i F^j}}_{D,1_Z,3}|=p^{2n/d}$. 
This discussion proves that, if $\nu=1_Z$, then for every $H\leq A$ we
have
\begin{equation}
\label{eq:int2}
|\cal O_{D,1_Z,3}^H|=|\cal O'^H_{D,1_Z,3}|.
\end{equation}
Now, if $\nu=\varepsilon^m$ with $m=\pm 1$, 
then $A_{\nu}=\cyc{\sigma F}$ and for every $H=\cyc{\sigma^i F^{n/d}}\leq
A_{\nu}$, the same
argument shows that 
\begin{equation}
\label{eq:int3}
|\cal O_{D,\varepsilon^m,3}^H|=p^{2n/d}=|
\cal O'^H_{D,\varepsilon^m,3}|.
\end{equation}
So, this proves that for every $\nu\in\Irr(\Zz(\Galg)^{F'})$, the sets
$\cal O_{D,\nu,3}$ and $\cal O'_{D,\nu,3}$ are $A_{\nu}$-equivalent. 

Suppose now that $k=1$. Note that $s(\Galg^{*F'^*})$ and $\cal O_{D}$
are $A$-equivalent. Let $d$ be a divisor of $n$. Then, we have
$|s(\Galg^{*F'^*})^{\cyc{\sigma^{*i}F^{*n/d}}}|=|s(\Galg^{*F^{*n/d}})|$,
because the set of representatives $\cal T$ of $F'$-stable semisimple
classes of $\Galg^*$ can be chosen such that if the class of $t\in\cal
T$ is $\sigma^{*i} F^{*n/d}$-stable, then $\sigma^{*i} F^{*n/d}(t)=t$. We then
conclude using the fact that a power of $\sigma^{*i} F^{*n/d}$ equals $F'$
and with Equation~(\ref{eq:paramclass}). So, by~\cite[1.1]{Br8}, we
deduce that 
\begin{equation}
\label{eq:calinter}
\left|s(\Galg^{*F'^*})^{\cyc{\sigma^{*i} F^{*n/d}}}\right|=\left\{
\begin{array}{ll}
p^{n|\Delta|/d}+2p^{2n/d}&\textrm{if } \Zz(\Galg)^{\sigma^i
F^{n/d}}=\Zz(\Galg)\\
p^{n|\Delta|/d}&\textrm{otherwise}
\end{array}\right..
\end{equation}
Furthermore, if $t\in\cal T$ is chosen $F^{*n/d}$-stable 
when the
class of $t$ in $\Galg^*$ is $F^{*n/d}$-stable,
then we have $|
s(\Galg^{*F'^*})^{\cyc{\sigma^*}\times\cyc{F^{*n/d}}}|=
|s(\Galg^{*F^{n/d}})^{\sigma^*})|$.
Thanks to Theorem~\ref{karine}, we deduce that
\begin{equation}
\left|s(\Galg^{*F'^*})^{\cyc{\sigma^*}\times\cyc{F^{*n/d}}}\right|
=p^{n|\Delta_{\sigma}|/d}.
\label{eq:calinter2}
\end{equation}
Now, using the case $k=3$, the fact that $\cal O_D=\cal O_{D,1}\sqcup
\cal O_{D,3}$ and Equations~(\ref{eq:calinter}) and~(\ref{eq:calinter2}),
we deduce that
\begin{equation}
\left|\cal
O_{D,1}^{\cyc{\sigma^{*i}F^{*n/d}}}\right|=p^{n|\Delta|/d}-p^{2n/d}\quad\textrm{and}
\quad
\left|\cal
O_{D,1}^{\cyc{\sigma^{*}}\times\cyc{F^{*n/d}}}\right|=
p^{n|\Delta_{\sigma}|/d}-p^{2n/d}.
\label{eq:calinter3}
\end{equation}
 
Moreover, note that the argument at the beginning of the proof shows
that the set $\cal O'^{\cyc{\sigma}\times\cyc{F^{n/d}}}_D$ is in
bijection with $\Irr_s((\Balg^{\sigma})^{F^{n/d}})$, which has
$p^{d|\Delta_{\sigma}|}$ elements by~\cite[Proposition 3]{Br6}. 
The set $\cal O'^{\cyc{\sigma^i F^{n/d}}}_D$ is in bijection with
$\Irr_s(\Balg^{\sigma^i F^{n/d}})$, which has
$p^{n|\Delta|/d}+2p^{2n/d}$
elements if $\sigma^i F^{n/d}$ acts trivially on $\Zz(\Galg)$, and
$p^{d|\Delta|}$ elements otherwise. Since $\cal O'_D=\cal
O'_{D,1}\sqcup\cal O'_{D,3}$, we deduce from Equations~(\ref{eq:int1}),
(\ref{eq:int2}) and~(\ref{eq:int3}) that
%
\begin{equation}
\left|\cal
O'^{\cyc{\sigma^{i}F^{n/d}}}_{D,1}\right|=p^{n|\Delta|/d}-p^{2n/d}\quad
\textrm{and}
\quad
\left|\cal
O'^{\cyc{\sigma^{}}\times\cyc{F^{n/d}}}_{D,1}\right|=
p^{n|\Delta_{\sigma}|/d}-p^{2n/d}.
\label{eq:int4}
\end{equation}
Equations~(\ref{eq:calinter3}) and~(\ref{eq:int4}) prove that the sets
$\cal O_{D,1}$ and $\cal O'_{D,1}$ are $A$-equivalent.  So,
by~\cite[Theorem 1.1]{Br8}, the sets $\Irr_s(\Galg^{F'},\nu)$ and
$\Irr_s(\Balg^{F'},\nu)$ are in bijection. Since $\cal O_{D,\nu,3}$ and
$\cal O'_{D,\nu,3}$ are in bijection by Equations~(\ref{eq:int2})
and~(\ref{eq:int3}), we deduce that $\cal O_{D,\nu,1}$ and $\cal
O'_{D,\nu,1}$ have the same cardinal.  Moreover, if $H$ is not a
subgroup of $\cyc{\sigma F}$, then every $H$-stable character of
$\Irr_s(\Galg^{F'})$ (resp. $\Irr_s(\Balg^{F'})$) lies over $1_Z$. Now,
let $H=\cyc{\sigma^i F^{n/d}}$ be a subgroup of $\cyc{\sigma F}$. We
consider the norm map $N_{F'/\sigma^i F^{n/d}}:\Zz(\Galg)\rightarrow
\Zz(\Galg)$, which is well-defined because $(\sigma^iF^{n/d})^n=F'^n$.
By~\cite[Lemma 5.8, Lemma 5.9]{BrHi}, if $N_{F'/\sigma^iF^{n/d}}$ is
surjective, then every set $\cal O'_{D,\nu,1}$ and $\cal O'_{D,\nu,1}$
contains $\frac{1}{3}(p^{n|\Delta|/d}-p^{2n/d})$ characters invariant
under $\sigma^iF^{n/d}$.
Otherwise (i.e., when $N_{F'/\sigma^iF^{n/d}}$ is trivial), the
$\sigma^iF^{n/d}$-stable characters of $\Irr_s(\Galg^{F'})$ and
$\Irr_s(\Balg^{F'})$ lie over $1_Z$. So, this proves that, for every
subgroup $H$ of $A_{\nu}$, we have
$$|\cal O^H_{D,\nu,1}|=|\cal O'^H_{D,\nu,1}|.$$
Therefore, the sets $\cal O_{D,\nu,1}$ and $\cal O'_{D,\nu,1}$ are
$A_{\nu}$-equivalent, as required.

Now, we suppose that $\Zz(\Galg)^F$ and $\Zz(\Galg)^{F'}$ are trivial,
that is, $p\not\equiv 1\mod 3$ and $q\not\equiv 1\mod 3$.
Then $D$ is trivial and we conclude using the argument of
Equations~(\ref{eq:calinter}) and~(\ref{eq:calinter2}), and the
discussion following Equation~(\ref{eq:calinter3}).

Finally, suppose that $\Zz(\Galg)^F=\Zz(\Galg)$, i.e. $p\equiv q\equiv
1\mod 3$. Then
$\Zz(\Galg)^{F'}=\Zz(\Galg)$. If $n$ is odd, we remark that if $X$ is a
$\cyc{\sigma F^{n/d}}$-set and an $\cyc{F'}$-set, and if $x$ is fixed by
$\sigma F^{n/d}(x)$ and $F'$, then $F^{n/d}(x)=x$ and $\sigma(x)=x$. In
particular, we can compare $|\Irr_{s}(\Galg^{F'})^{\sigma F^{n/d}}|$
with the number of semisimple classes of $\Galg^{F^{n/d}}$ fixed by
$\sigma$ (there are $p^{n|\Delta_{\sigma}|/d}$ such classes by
Lemma~\ref{gamma1} and Theorem~\ref{karine}), and
$|\Irr_{s}(\Balg^{F'})^{\cyc{\sigma F^{n/d}}}|$ with $|\Irr_s(
(\Balg^{\sigma})^{F^d})|=p^{n|\Delta_{\sigma}|/d}$. If $n$ is even, then
$(\sigma^iF^j)=F'$ and we are in the same situation as the first case of
the proof. Using the same argument as above we deduce the numbers of
Table~\ref{tab:preuve}. We set
$a_{F^{n/d}}= b_{F^{n/d}}=p^{n|\Delta|/d}-p^{2n/d}$ when the norm map
$N_{F'/F^{n/d}}$ is surjective, and $a_{F^{n/d}}=
p^{n|\Delta|/d}-3p^{2n/d}$ and $b_{F^{n/d}}=0$ otherwise.
\begin{table}
$$
\renewcommand{\arraystretch}{1.6}
\begin{array}{c|c|c|c|c}
H&n&\cyc{\sigma}\times\cyc{F^{n/d}} &\cyc{F^{n/d}} &\cyc{\sigma F^{n/d}} \\
\hline
|\cal O^H_{D,1_Z,3}|&&p^{2n/d}&p^{2n/d}&p^{2n/d}\\
|\cal O^H_{D,\varepsilon,3}|,\mbox{\scriptsize $\varepsilon\neq 1_Z$}&
&0&p^{2n/d}&0\\
|\cal O^H_{D,1_Z,1}|&\textrm{odd}&p^{n|\Delta_{\sigma}|/d}
-p^{2n/d}&a_{F^{n/d}}
&p^{n|\Delta_{\sigma}|/d}-p^{2n/d}\\
|\cal O^H_{D,1_Z,1}|&\textrm{even}&p^{n|\Delta_{\sigma}|/d} -p^{2n/d}
&a_{F^{n/d}}
&p^{n|\Delta|/d}-p^{2n/d}\\
|\cal O^H_{D,\varepsilon,1}|,\mbox{\scriptsize $\varepsilon\neq 1_Z$}
&&0&
b_{F^{n/d}}
&0\\
\hline
|\cal O'^H_{D,1_Z,3}|&&p^{2n/d}&p^{2n/d}&p^{2n/d}\\
|\cal O'^H_{D,\varepsilon,3}|,\mbox{\scriptsize $\varepsilon\neq 1_Z$}&
&0&p^{2n/d}&0\\
|\cal O'^H_{D,1_Z,1}|&\textrm{odd}&p^{n|\Delta_{\sigma}|/d}
-p^{2n/d}&
a_{F^{n/d}}
&p^{n|\Delta_{\sigma}|/d}-p^{2n/d}\\
|\cal O'^H_{D,1_Z,1}|&\textrm{even}&p^{n|\Delta_{\sigma}|/d}
-p^{2n/d}&
a_{F^{n/d}}
&p^{n|\Delta|/d}-p^{2n/d}\\
|\cal O'^H_{D,\varepsilon,1}|,\mbox{\scriptsize $\varepsilon\neq 1_Z$}
&&0&
b_{F^{n/d}}
&0\\
\end{array}$$
\caption{Case when $\Zz(\Galg)^F=\Zz(\Galg)$}
\label{tab:preuve}
\end{table}
It follows that $\cal O_{D,\nu,k}$ and $\cal O'_{D,\nu,k}$ are
$A_{\nu}$-equivalent. This proves the claim.
\end{proof}

\subsection{Inductive McKay condition}

\begin{lemma}
Let $H=GC$ be a finite central product with $Z=G\cap C$. Suppose that
$C$ is abelian and that
$\tau\in\operatorname{Aut}(H)$ acts on $G$ and $C$. For
$\chi\in\Irr(G)^{\tau}$, write $\nu=\Res_Z^G(\chi)$. If $\nu$ extends to
a $\tau$-stable character of $C$, then we have
$$|\Irr(H|\chi)^{\tau}|=|(C/Z)^{\tau}|.$$
\label{lemme:cp}

\end{lemma}

\begin{proof}
Recall that, for $\theta\in\Irr(H)$, if we write
$\gamma=\Res_{Z}^H(\theta)$, then there are 
unique $\theta_1\in \Irr(G|\gamma)$ and $\theta_2\in\Irr(C|\gamma)$ such
that $\theta=\theta_1\cprod\theta_2$, where $\cprod$ is the dot product,
defined by $\theta_1\cprod\theta_2(g_1g_2)= \theta_1(g_1)\theta_2(g_2)$ for all
$g_1\in G$ and $g_2\in C$.
Let $\chi\in\Irr(G)^{\tau}$. Write $\nu=\Res_{Z}^H(\chi)$. Then we have
$$\Irr(H|\chi)=\{\chi\cprod \theta\,|\,\theta\in\Irr(C|\nu)\}.$$
Since $\chi$ is $\tau$-stable, so is $\nu$ and $\Irr(C|\nu)$ is
$\tau$-stable. Furthermore, $\chi\cprod\theta$ is $\tau$-stable if and
only if $\theta$ is $\tau$-stable. In particular, the set
$\Irr(H|\chi)^{\tau}$ is parametrized by $\Irr(C|\nu)^{\tau}$.
Moreover, $C$ is abelian, then $\nu$ extends to a linear character of $C$
(denoted by the same symbol)~\cite[5.4]{Isaacs} and by assumption, $\nu$ can
be supposed to be $\tau$-stable. It follows that the map
$g_{\nu}:\Irr(C|1_Z)\rightarrow \Irr(C|\nu),\,\theta\mapsto\theta\nu$ is a
bijection such that $\Irr(C|\nu)^{\tau}=g_{\nu}(\Irr(C|1_Z)^{\tau})$
(because if $\nu$ is $\tau$-stable, so is $\nu^{-1}$).
But we can identify $\Irr(C|1_Z)$ with $\Irr(C/Z)$ and the action of
$\tau$ on these sets is compatible. Hence, we have
$|\Irr(C|1_Z)^{\tau}|=|\Irr(C/Z)^{\tau}|$ and the result follows
from~\cite[6.32]{Isaacs}.
\end{proof}
We recall that if $N$ is a normal subgroup of $G$, then we can associate
to every $G$-invariant irreducible character of $N$ an element
$[\chi]_{G/N}$ of the second cohomology group $H^2(G/N,\C^{\times})$ of
$G/N$. 
\begin{theorem}Let $p>3$ be a prime number and $q$ a $p$-power.
Then the finite simple group $E_6(q)$ is ``good'' for the prime $p$.
\label{E6bon}
\end{theorem}

\begin{proof}
Let $X$ be a simple group of type $E_6$ with parameter $q=p^n$. In order
to prove that $X$ is ``good'' for $p$, we will show that $X$ satisfies
the properties (1)--(8) of~\cite[\S10]{IMN}. Let $\Galg$ be a simple
simply-connected group of type $E_6$ defined over $\F_q$ and with split
Frobenius map $F':\Galg\rightarrow\Galg$. Recall that
$X=\Galg^{F'}/\Zz(\Galg)^{F'}$ and $\Galg^{F'}$ is the universal cover
of $X$. Moreover, the subgroup $\Ualg^F$ is a $p$-Sylow subgroup of
$\Galg^{F'}$ with normalizer $\Balg^{F'}$. We set $M=\Balg^{F'}$ and we
will show that $M$ has the required properties. By~\cite[Lemma 5]{Br6}, we
have $\Irr_s(\Galg^{F'})=\Irr_{p'}(\Galg^{F'})$ and
$\Irr_s(\Balg^{F'})=\Irr_{p'}(\Balg^{F'})$.
Let $\nu\in\Irr(\Zz(\Galg)^{F'})$ and $\cal A_{\nu}=D\rtimes A_{\nu}$ as
above. By Theorem~\ref{equiv1}, there is an
$\cal A_{\nu}$-equivariant bijection
$\Phi_{\nu}:\Irr_{p'}(\Galg^{F'}|\nu)\rightarrow\Irr_{p'}(\Balg^{F'}|\nu)$.
Thus, Properties (1)--(4) of~\cite[\S10]{IMN} are satisfied.
For $\chi\in\Irr_{p'}(\widetilde{\Galg}^{F'}|\nu)$, we define
$$G_{\chi}=G\rtimes\Stab_{A_{\nu}}(\chi)\quad\textrm{and}\quad
G'_{\chi}=G'\rtimes\Stab_{A_{\nu}}(\Phi_{\nu}(\chi)),$$
where $G=\Galg^{F'}$ (resp. $G'=\Balg^{F'}$) if $\chi$ is not
$D$-stable, and $G=\widetilde{\Galg}^{F'}$ (resp.,
$G'=\widetilde{\Balg}^{F'}$) otherwise, where $\widetilde{\Galg}$ is the
connected group defined in Equation~(\ref{eq:Gtilde}) with the
convention of~\S\ref{sec:preliminaire}. We have $\Zz(\Galg^{F'})\leq
\Zz(G_{\chi})$ and $\Stab_{D\rtimes A}(\chi)$ is induced by the
conjugation action of $\operatorname{N}_{G_{\chi}}(\Ualg^{F'})$.
So, the property~(5) of~\cite[\S10]{IMN} holds.
We set $C=\Zz(G)$. Then $C=\Cen_{G_{\chi}}(\Galg^{F'})$ and the
property~(6) of~\cite[\S10]{IMN} is true. 

In order to prove the properties~(7) and~(8) of~\cite[\S10]{IMN}, we will
first show that $\chi$ and $\Phi_{\nu}(\chi)$ extend to $G_{\chi}$ and
$G_{\chi}'$, respectively. 
If $\chi$ is not $D$-stable, then $\Stab_A(\chi)\leq \cyc{F,\sigma}$. If
$\Stab_{A}(\chi)$ is cyclic, then $\chi$ extends to $G_{\chi}$
by~\cite[11.22]{Isaacs}. Otherwise, $\Stab_A(\chi)=\cyc{\sigma,
F^{n/d}}$ for some divisor $d$ of $n$, and by
Proposition~\ref{extension}, the extensions of $\chi$ to
$\Galg^{F'}\rtimes{\sigma}$ are $F^{n/d}$-stable, and thus extend to
$G_{\chi}$ by~\cite[11.22]{Isaacs}. So, $\chi$ is extendible to
$G_{\chi}$.

Suppose now that $D$ is not trivial and
 $\chi$ is $D$-stable. In particular, $\chi$ is
extendible to $\widetilde{\Galg}^{F'}$.  Write $A_{\chi}=\Stab_A(\chi)$.
We will prove that $\chi$ extends to an $A_{\chi}$-stable character of
$\widetilde{\Galg}^{F'}$. First, we suppose that $\sigma\in A_{\chi}$
and we write $H=GC$. Recall that $H$ is a central product and that
$C=\Zz(\widetilde{\Galg}^{F'})$. By~\cite[\S6.B]{BonnafeAn}, $H$ has
index $|D|$ in $\widetilde{\Galg}^{F'}$. 
Note that $\chi$ is over $1_{Z}$.
Since $1_C$ lies in
$\Irr(C|1_Z)$, we deduce from Lemma~\ref{lemme:cp} that the
character $\chi\cprod 1_{C}$ is $A_{\chi}$-stable. Moreover,
Gallagher's theorem implies that the elements of
$E=\Irr(\widetilde{\Galg}^{F'}|\chi\cprod 1_{C})$ extend $\chi$ and
$|E|=|D|=3$. Moreover, $E$ is $A_{\chi}$-stable. Denote by
$\rho:A_{\chi}\rightarrow \mathfrak{S}_E$ the homomorphism of groups induced
by this operation. Suppose that $\sigma\in A_{\chi}$.
Note that $\sigma$ does not act trivially on $E$
(because if $\sigma$ fixes a character of $E$, then the action of
$\sigma$ on $E$ is equivalent to the action of $\sigma$ on
$\widetilde{\Galg}^{F'}/H$, which is not trivial). So, $\rho(\sigma)\neq
1$ and $\rho(\sigma)$ has order $2$. Thus, $\sigma$ has to fix a
character of $E$. Suppose now that there is $F^i\in A_{\chi}$. Since
$\sigma$ and $F^{i}$ commute, we deduce that $\rho(F^i)$ centralizes
$\rho(\sigma)$. But $\mathfrak{S}_E\simeq\mathfrak{S}_3$ and the
centralizer of $\rho(\sigma)$ in $\mathfrak{S}_E$ has order $2$. Thus,
there is an $A_{\chi}$-stable character $\widetilde{\chi}$ in $E$. 
By~\cite[1.31]{DMnonconnexe}, we have
$\widetilde{\Galg}^{\sigma}=\rad(\widetilde{\Galg})^{\sigma}D(\Galg)^{\sigma}$,
because $\sigma$ has order $2$ and $|\Zz(D(\widetilde{\Galg}))|=3$.
Since $\rad(\widetilde{\Galg})=\Talg_0$ (cf. Equation~(\ref{eq:radE6}))
has dimension $1$, we deduce that $\sigma$ acts by inversion on
$\Talg_0$ and $|\Talg_0^{\sigma}|=2$. Hence,
$\widetilde{\Galg}^{\sigma}$ is disconnected with connected component a
simple group of type $F_4$ (by Lemma~\ref{GtildeE6}) of index $2$. We
deduce that $|\Zz(\widetilde{\Galg}^{\sigma})|=2$.
Denote by $\widetilde{\Psi}$ the character of $\widetilde{\Galg}$
constructed in Equation~(\ref{eq:defdualgelfand}). The character
$\widetilde{\chi}$ extends to a character $E(\widetilde{\chi})$ of
$\widetilde{\Galg}^{F'}\rtimes\cyc{\sigma}$, and
$\cyc{\widetilde{\Psi},E(\widetilde{\chi})}_{\widetilde{\Galg}^{F'}
\cdot\sigma}=\pm 1$. Moreover, by~\cite[Proposition 7.2]{Sorlin}
$\widetilde{\Psi}$ has non-zero values only
on $\widetilde{\cal U}_{\sigma}^{F'}$ (see the proof of
Proposition~\ref{extension} for the definition), which is the union of
the two classes $\widetilde{\cal U}_1$ and  $\widetilde{\cal U}_{-1}$ of
$\widetilde{\Galg}^{F'}\rtimes\cyc{\sigma}$ (by \cite[Proposition
8.1]{Sorlin}, because $\sigma$ is semisimple). Then
$E(\widetilde{\chi})$ has a non-zero value on at least 
$\widetilde{\cal U}_1$ or $\widetilde{\cal U}_{-1}$. But $F^i$ fixes
these two classes (because there is an $F^i$-stable element in
$\widetilde{\cal U}_{\sigma}^{F'}$. So, its
$\widetilde{\Galg}^{F'}\rtimes\cyc{\sigma}$-class is $F^i$-stable
and the other class has to be also $F^i$-stable). Then the argument of
Proposition~\ref{extension} shows that $E(\widetilde{\chi})$ is
$F^i$-stable. It follows that $E(\widetilde{\chi})$ extends to
$(\widetilde{\Galg}^{F'}\rtimes\cyc{\sigma})\rtimes\cyc{F^i}$
by~\cite[11.22]{Isaacs}. This proves that $\chi$ extends to $G_{\chi}$.

Now, we suppose that $\sigma\notin A_{\chi}$. Then
$A_{\chi}=\cyc{\sigma^i F^{n/d}}$ for some $d\neq n$. 
Moreover, as we have seen in the proof of Theorem~\ref{equiv1}, we have
$F'=(\sigma^i F^{n/d})^n$ (if not, $\sigma$ has to fix $\chi$).
Then, there is a semisimple element $s$ of $\Galg^{*F'^*}$ such that
$\chi=\rho_{s,1}$. Since $A_{\Galg^*}(s)^{F'^*}$ is trivial, by
the Lang-Steinberg theorem, we can suppose that $s$ is chosen to be 
$\sigma^{*i}F^{*n/d}$-stable. Let $\widetilde{s}$ be an $\sigma^{*i}
F^{*n/d}$-stable semisimple element of $\widetilde{\Galg}^*$ such that
$i^*(\widetilde{s})=s$ (such elements exist by the Lang-Steinberg theorem,
because $\ker(i^*)$ is connected and $\sigma^{*i}F^{*n/d}$-stable). 
Note that $\widetilde{s}\in\widetilde{\Galg}^{*F'^*}$, and since
$\widetilde{s}$ is $\sigma^{*i}F^{*n/d}$-stable, the character 
$\rho_{\widetilde{s}}$ is $\sigma^{i}F^{n/d}$-stable. Moreover,
$\rho_{\widetilde{s}}$ extends $\chi$.
Since $G_{\chi}$ is a cyclic extension of $\widetilde{\Galg}^{F'}$, we
deduce from~\cite[11.22]{Isaacs} that $\chi$ extends to $G_{\chi}$, as
required.

Write $\chi'=\Phi_{\nu}(\chi)$. Then there are $J\subseteq \Delta$,
$z\in H^1(F',\Zz_J)$ and $\psi\in\Irr(\Talg_J)^{F'}$ with
$\Res_{\Zz(\Galg)^{F'}}^{\Talg_J^{F'}}(\psi)=\nu$, such that
$\chi'=\chi_{J,z,\psi}$.
Suppose that $\chi'$ is $D$-stable. Then $z=1$ and $\Talg_J$ is connected.
Write
$\widetilde{\Talg}_J=\Cen_{\widetilde{\Talg}^{F'}}(\phi_J)$, where
$\phi_J$ is chosen as in Convention~\ref{conv3}. Note that
$\widetilde{\Talg}_J$ is a torus because the center of
$\widetilde{\Galg}$ is connected.  
Then $\Talg_J$ is a subtorus of $\widetilde{\Talg}_J$
and by~\cite[0.5]{DM}, there is an
$F'$-stable subtorus $\Talg'$ of $\widetilde{\Talg}_J$ such that
$\widetilde{\Talg}_J=\Talg_J\cdot\Talg'$ (as direct product). 
By~\cite[6.17]{Isaacs} and
the construction of Equation~(\ref{eq:defcharB}) applied to $\Balg$ and
$\widetilde{\Balg}$, we deduce that
$$\Irr(\Balg^{F'}|\chi_{J,1,\psi})=\{\widetilde{\chi}_{J,1,\psi\otimes\mu}
\,|\,\mu\in\Irr(\Talg'^{F'})\},$$
where $\widetilde{\chi}_{J,1,\psi\otimes\mu}$ denotes the character of
$\widetilde{\Balg}^{F'}$ defined in Equation~(\ref{eq:defcharB}).
Note that if $\Talg_J$ is $\sigma^iF^j$-stable (for any $i$, $j$), then
$\Talg'$ is $\sigma^iF^j$-stable. Write $A_{\chi'}=\Stab_A(\chi')$. It
follows that the character $\widetilde{\chi}_{J,1,\psi\otimes
1_{\Talg'^{F'}}}\in\Irr(\widetilde{\Balg}^{F'})$ is an
$A_{\chi'}$-stable extension of $\chi'$.
For $\tau\in A_{\chi'}$, $\phi_J$ and $\psi$ are
$\tau$-stable. Write $\widetilde{\psi}=\psi\otimes 1_{\Talg'^{F'}}$.
Then $\widetilde{\Talg}_J^{F'}$ and 
$\widehat{\phi}_{J}\otimes\widetilde{\psi}$ 
are $A_{\chi'}$-stable. Hence, 
as linear character, $\widehat{\phi}_{J}\otimes\widetilde{\psi}$ extends to a
linear character $E(\widehat{\phi}_{J}\otimes\widetilde{\psi})$ of
$(\Ualg^{F'}\rtimes\widetilde{\Talg}_J^{F'})\rtimes
A_{\chi'}$, and by~\cite[(5.6) p.74]{Isaacs}, we have
\begin{equation}
\label{eq:extB}
\widetilde{\chi}_{J,1,\widetilde{\psi}}=
\Res_{\widetilde{\Balg}^{F'}}^{\widetilde{\Balg}^{F'}\rtimes
A_{\chi'}}
\left(\Ind_{(\Ualg^{F'}\rtimes\widetilde{\Talg}^{F'}_J)\rtimes A_{\chi'}}^{
\widetilde{\Balg}^{F'}\rtimes
A_{\chi'}}E(\widehat{\phi}_{J}\otimes\widetilde{\psi})\right).
\end{equation}
Hence, $\chi'$ extends to $G'_{\chi}$.
Suppose that $\chi'$ is not $D$-stable. If $A_{\chi'}$ is cyclic, then
$\chi'$ extends to $G_{\chi}'$ by~\cite[11.22]{Isaacs}. Otherwise,
$z=1$ and we can show that $\chi'$ extends to $A_{\chi'}$, because 
we can construct an extension of $\chi_{J,1,\psi}$ to
$\Balg^{F'}\rtimes A_{\chi'}$ as in Equation~(\ref{eq:extB}).

We now will prove the properties~(7) and~(8) of~\cite[\S10]{IMN}.  If $\chi$
is not $D$-stable, then $C=\Zz(\Galg)^{F'}$ and we choose $\gamma=\nu$.
Then $\chi\cprod\gamma=\chi$ and
$\Phi_{\nu}(\chi)\cprod\gamma=\Phi_{\nu}(\chi)$ and by the preceding
discussion, these characters extend to $G_{\chi}$ and $G_{\chi}'$.  
Then by~\cite[11.7]{Isaacs}, we have
$$[\chi]_{G_{\chi}/\Galg^{F'}}=[\Phi_{\nu}(\chi)]_{G_{\chi}'/\Balg^{F'}},$$
and Properties~(7) and~(8) of~\cite[\S10]{IMN} are proven in this case.

Suppose now that $D$ is not trivial and that $\chi$ is $D$-stable. If
$\nu=1_Z$ then we set $\gamma=1_C$, which is $G_{\chi}$ and
$G_{\chi}'$-stable. Moreover, as we have seen, the characters
$\chi\cprod 1_C$ and $\Phi_{\nu}(\chi)\cprod 1_C$ extend to $G_{\chi}$
and $G_{\chi}'$. If $\nu\neq 1_Z$, then $G_{\chi}=\cyc{F''}$ with
$F''=\sigma^i F^{n/d}$ for some $d\neq 1$ and $F''^n=F'$. Moreover,
since $\nu\neq 1$ is $F''$-stable, it follows that $F''$ acts trivially
on $\Zz(\Galg)^{F'}$. Note that
$\Phi_{\nu}(\chi)$ is $D$-stable and extends to $G_{\chi}'$. Denote by
$\widetilde{\chi}'$ such an extension and write
$$\gamma=\Res_{\Zz(\widetilde{\Galg})^{F'}}^{G_{\chi}'}(\widetilde{\chi}').$$
Then $\gamma$ is $G_{\chi}'$-stable and $G_{\chi}$-stable and lies over
$\nu$.
Furthermore, by~\cite[6.17]{Isaacs}, we have
\begin{equation}
\left |\Irr(\widetilde{\Galg}^{F'}|\chi)^{F''}\right|
=\sum_{\theta\in\Irr(\Galg^F\,
C|\chi)^{F''}}\left|\Irr(\widetilde{\Galg}^{F'}|\theta)^{F''}\right|.
\label{eq:inn}
\end{equation}
As we have seen previously, $\chi=\rho_{s,1}$ for some $F''$-stable
semisimple element $s\in\Galg^{*F'^*}$ and there is an $F''$-stable
semisimple element $\widetilde{s}\in\widetilde{\Galg}^{*F'^*}$ such that
$i^*(\widetilde{s})=s$. Then by~\cite[(15.9)]{BonnafeAn}, we have
$\Irr(\widetilde{\Galg}^{F'}|\chi)=\{\rho_{\widetilde{s}z,1}\,|\,z\in
C\}$ and~\cite[Lemma 8.3]{BonnafeAn} implies that
this set is in bijection with $C$ because
$A_{\Galg^*}(s)^{F'}$ is trivial. Moreover, since $\widetilde{s}$ is
$F''$-stable, these operations are $F''$-equivalent and
$|\Irr(\widetilde{\Galg}^{F'},|\chi)^{F''}|=|C^{F''}|$.
Now, Lemma~\ref{lemme:cp} implies that $|\Irr(\Galg^F\,
|\chi)^{F''}|=|C^{F''}|/3$, because $F''$ acts trivially on
$\Zz(\Galg^F)=\Galg^F\cap C$. Since
$|\Irr(\widetilde{\Galg}^{F'}|\theta)^{F''}|\leq 3$ (because $\Galg^F C$
has index $3$ in $\widetilde{\Galg}^{F'}$ and by~\cite[6.17]{Isaacs}),
we deduce from Equation~(\ref{eq:inn}) that
$|\Irr(\widetilde{\Galg}^{F'}|\theta)^{F''}|$ has to be equal to $3$. 
But $\chi\cdot\gamma\in\Irr(\Galg^F\,C|\chi)^{F''}$. Thus,
$\chi\cdot\gamma$ extends to an $F''$-stable character of
$\widetilde{\Galg}^{F'}$ and~\cite[11.22]{Isaacs} implies that
$\chi\cdot\gamma$ extends to $G_{\chi}$.
We conclude using~\cite[11.7]{Isaacs} that 
$$[\chi\cprod\gamma]_{G_{\chi}/\Galg^{F'}
C}=[\Phi_{\nu}(\chi)\cprod\gamma]_{\Galg_{\chi}'/\Galg^{F'}\,C}.$$
Hence, the properties~(7) and~(8) of~\cite[\S10]{IMN} hold, as required.
\end{proof}

\begin{theorem}
Let $p>3$ be a prime number and $q$ a $p$-power.
Then the finite simple group $^2E_6(q)$ is ``good'' for the prime $p$.
\label{2E6bon}
\end{theorem}

\begin{proof}
We set $F'=\sigma F^n$. Then $\Galg^{F'}$ is the universal cover of
$^2E_6(q)$, and the outer automorphism group of $\Galg^{F'}$ is
$D\rtimes \cyc{F}$, where $F$ acts on $\Galg^{F'}$ as an automorphism of
order $2n$.
For $\omega\in\Delta/\rho$, we set
$\Xalg_{\omega}=\prod_{\alpha\in\omega}\Xalg_{\alpha}$. Note that
$\Xalg_{\omega}$ is a subgroup of $\Ualg$, because $\Galg$ is of type
$E_6$ and if the roots $\alpha$ and $\rho(\alpha)$ are distinct, then
they are orthogonal. Moreover, by~\cite[\S8.1]{carter2} we have
$\Ualg_1^{F'}\simeq\prod_{\omega\in\Delta/\rho} \Xalg_{\omega}^{F'}$ (as
direct product) and can
construct the Gelfand Graev characters, the regular characters 
and the semisimple characters of $\Galg^{F'}$ as
in~\S\ref{subsectionSemi}. Note that the analogue of Theorem~\ref{decrsemi} is valid
(see~\cite[Proposition 15.3,\,
Corollaire 15.14]{BonnafeAn}) and the regular character $\phi_1$ of
$\Irr_l(\Ualg^{F'})$ can be chosen $F$-stable by~\cite[6.32]{Isaacs}.
In particular, if $s$ is an $F^*$-stable element of $\Galg^{*F'^*}$,
then for every $z\in H^1(F',\operatorname{Z}(\Galg))$, we have
$${}^{F^i}\rho_{s,\,\widehat{\omega}_s^0(z)}=\rho_{s,\,
\widehat{\omega}_s^0(F^i(z))}.$$
For every $J\subseteq \Delta/\rho$, we write $\omega_J$ for the
corresponding $\widetilde{\Talg}^{F'}$-orbit in $\Irr_l(\Ualg^{F'})$ and we
choose a representative $\phi_J\in\omega_J$ as in
Convention~\ref{conv3}. Then we can define
$\chi_{J,z,\psi}\in\Irr_s(\Balg^{F'})$ as in
Equation~(\ref{eq:defcharB}), which satisfies Lemma~\ref{imgautoB}.

Let $E$ be an $F$- and $F'$-set, and let $x\in E$ be such
that $F^k(x)=x$ and $F'(x)=x$. Denote by $d$ the order of $F^k$
(viewed as an automorphism of $\Galg^{F'}$). If $d$
is even (resp. odd), then this is equivalent to $\sigma(x)=x$ and
$F^k(x)=x$ (resp. $\sigma F^{n/d}(x)=x$). 
Moreover, note that if $\Zz(\Galg)^{F'}=\Zz(\Galg)$,
then $\Zz(\Galg)^{F}=\{1\}$ and $n$ is odd.
Using these facts, we can prove in a similar way as in the proof of
Theorem~\ref{equive6}, that 
for every $\nu\in\Irr(\Zz(\Galg)^{F'})$ and
$A_{\nu}=\Stab_{\cyc{F}}(\nu)$,
there is an $D\rtimes
A_{\nu}$-equivariant bijection between $\Irr_s(\Galg^{F'}, \nu)$ and
$\Irr_s(\Balg^{F'})$.
Finally, we prove that the properties~(5)-(8) of~\cite[\S10]{IMN} are
satisfied as in the proof of Theorem~\ref{E6bon}.
\end{proof}

\begin{remark}
This method is not sufficient to show the statement for $p\in\{2,3\}$.
Indeed, we need the assumption that $p$ is a good prime for $\Galg$ to
apply the ``relative'' version of the McKay Conjecture 
proved in~\cite[Theorem 1.1]{Br8}, and~\cite[Proposition 8.1(ii)]{Sorlin}.
\end{remark}

\begin{proposition}
\label{An}
Suppose that $X=\operatorname{PSL}_{\ell}(p^n)$ or
$X=\operatorname{PSU}_{\ell}(p^n)$ such that $p$ is odd and $\ell$ is
an odd prime number that not divides $p$. Then $X$ is ``good'' for $p$. 
\end{proposition}

\begin{proof}
We set $\widetilde{\Galg}=\operatorname{GL}_{\ell}(\overline{\F}_p)$ and
denote by $F$ the standard Frobenius map of $\widetilde{\Galg}$ which
acts by raising all entries of a matrix to the $p$-power. Write $F'=F^n$
and $\sigma$ for the non-trivial graph automorphism with respect to the
$F$-stable Borel subgroup of lower triangular matrices and the
$F$-stable maximal torus $\widetilde{\Talg}$ of diagonal matrices of
$\widetilde{\Galg}$. We set
$\Galg=\operatorname{SL}_{\ell}(\overline{\F}_p)$,
$\Talg=\widetilde{\Talg}\cap\Galg$ and
$\Balg=\widetilde{\Balg}\cap\Galg$.
Note that
$\widetilde{\Galg}^{\sigma}=\operatorname{GO}_{\ell}(\overline{\F}_p)$
and $\Galg^{\sigma}=\operatorname{SO}_{\ell}(\overline{\F}_p)$, which
implies that $\Zz(\Galg^{\sigma})=\{1\}$ and
$\Zz(\widetilde{\Galg}^{\sigma})$ has order $2$.
Denote by $\Delta=\{\alpha_1,\ldots,\alpha_{\ell-1}\}$ the set of simple
roots of $\Galg$ with respect to $\Talg$ and $\Balg$.
We set $\widetilde{\alpha}_{i}=\frac{1}{2}(\alpha_i+\alpha_{\ell -i})$
and $\Delta_{\sigma}=\{\widetilde{\alpha}_i\,|\, 1\leq 1\leq (\ell
-1)/2\}$. Then it is proven in~\cite[Lemma 4.4.7]{sol} that
$\Delta_{\sigma}$ is the set of simple roots of $\Galg^{\sigma}$ with
respect to $\Talg^{\sigma}$ and $\Balg^{\sigma}$. Moreover, if we set
$\widetilde{\Xalg}_{\widetilde{\alpha}_i}=\Xalg_{\alpha_i}.\Xalg_{\alpha_{\ell
-i}}$ for $1\leq i\leq (\ell -3)/2$ and $\widetilde{\Xalg}_{
\widetilde{\alpha}_{(\ell-1)/2}}=\cyc{\Xalg_{
\alpha_{(\ell-1)/2}},\Xalg_{\alpha_{(\ell+1)/2}}}$, then
the group $\widetilde{\Xalg}_{\widetilde{\alpha}_i}^{\sigma}$ is the
root subgroup corresponding to $\widetilde{\alpha}_i$; see the proof 
of~\cite[Lemma 4.4.7]{sol}.
Write $\widetilde{J}$ for the subset of $\Delta_{\sigma}$ associated
to a $\sigma$-stable subset $J$ of $\Delta$. Since
$\Zz(\Galg)^{\tau}=\{1\}$, it follows from the argument of the proof of
Theorem~\ref{equive6} that $\Talg_J^{\sigma}$ (here, $\Talg_J$ denotes
the radical of $\Lalg_J$ as before) is the radical of the
Levi subgroup $\Lalg_{\widetilde{J}}$ of $\Galg^{\sigma}$.

Therefore, by a similar argument to Lemma~\ref{equiv1} and
Theorem~\ref{equive6}, we show that for
$\nu\in\Irr(\Zz(\Galg)^{F'})$, there is an $A_{\nu}$-equivariant
bijection between $\Irr_s(\Galg^{F'}|\nu)$ and
$\Irr_s(\Balg^{F'}|\nu)$.

Now, we suppose that $\chi\in\Irr(\Galg^{F'})$ is
$\widetilde{\Talg}^{F'}$-stable and we denote by $A_{\chi}$ its
inertia subgroup in $\operatorname{Aut}(\widetilde{\Galg}^{F'})$. As
above, we write $E=\Irr(\widetilde{\Galg}^{F'}|\chi\cprod 1_C)$, where
$C=\Zz(\widetilde{\Galg}^{F'})$. Then $|E|=\ell$ and $\chi\cprod 1_C$ is
$A_{\chi}$-stable. So, $A_{\chi}$ acts on $E$.
Suppose that $\sigma\in A_{\chi}$. Since $\sigma$ has order $2$ and that
$\ell$ is odd, $\sigma$ fixes a character $\widetilde{\chi}$ of $E$.
Thus, by Clifford theory and by~\cite[6.32]{Isaacs}, the actions of
$\sigma$ on $E$ and on $\widetilde{\Galg}^{F'}/\Galg^FC$ are equivalent.
Since $\sigma$ acts by inversion on this group and has no non-trivial
fixed point (because $\ell$ is odd), we deduce that $\widetilde{\chi}$
is the unique $\sigma$-stable character of $E$.
Let $\tau\in A_{\chi}$. Then $\tau$ and $\sigma$ commute and
$\tau(\widetilde{\chi})$ is a $\sigma$-stable character of $E$. By
unicity, $\tau(\widetilde{\chi})=\widetilde{\chi}$, which proves that
$E$ has an $A_{\chi}$-stable element.  Finally, we conclude with a
similar argument to the proof of Theorem~\ref{E6bon}.  The proof for a
twisted Frobenius map is similar and the claim is proven.
\end{proof}

\noindent\textbf{Acknowledgements.}\quad
For valuable and helpful discussions on the paper~\cite{bonnafequasi}, I
wish to thank C\'edric Bonnaf\'e. For many fruitful conversations on the
subject, his precise reading of the manuscript and useful comments on the paper,
I sincerly thank Gunter Malle. I also thank Sebastian Herpel for
interesting discussions on algebraic group theory.  

 \bibliographystyle{abbrv}
\bibliography{references}

\end{document}